\documentclass[12pt]{amsart}
\usepackage{amssymb,amscd}
\usepackage{verbatim}

\usepackage{amsmath,amssymb,graphicx,mathrsfs}   
\usepackage{enumerate}
\usepackage[colorlinks=true,allcolors = blue]{hyperref} 

\usepackage{tikz}
\usetikzlibrary{matrix}

\usepackage[all]{xy}

\let\frak\mathfrak

\def\>{\relax\ifmmode\mskip.666667\thinmuskip\relax\else\kern.111111em\fi}
\def\<{\relax\ifmmode\mskip-.333333\thinmuskip\relax\else\kern-.0555556em\fi}
\def\vsk#1>{\vskip#1\baselineskip}
\def\vv#1>{\vadjust{\vsk#1>}\ignorespaces}
\def\vvn#1>{\vadjust{\nobreak\vsk#1>\nobreak}\ignorespaces}

  \let\ssize\scriptstyle
\let\sssize\scriptscriptstyle

\let\Medskip\medskip
\def\medskip{\par\Medskip}
\let\Bigskip\bigskip
\def\bigskip{\par\Bigskip}

\let\Maketitle\maketitle
\def\maketitle{\Maketitle\thispagestyle{empty}\let\maketitle\empty}

\newtheorem{thm}{Theorem}[section]
\newtheorem{cor}[thm]{Corollary}
\newtheorem{lem}[thm]{Lemma}

\newtheorem{defn}[thm]{Definition}

\theoremstyle{definition}                                  
\numberwithin{equation}{section}

\theoremstyle{definition}
\newtheorem*{rem}{Remark}
\newtheorem*{example}{Example}

\let\mc\mathcal
\let\nc\newcommand

\let\dl\delta
\let\Dl\Delta

\let\ka\kappa
\let\la\lambda
\let\La\Lambda

\let\phi\varphi
\let\si\sigma

\let\Om\Omega

\let\der\partial

\let\ox\otimes

\let\ge\geqslant
\let\geq\geqslant
\let\le\leqslant
\let\leq\leqslant

\let\on\operatorname
\let\bi\bibitem
\let\bs\boldsymbol

\def\C{{\mathbb C}}
\def\Z{{\mathbb Z}}
\def\R{{\mathbb R}}

\def\F{{\mathbb F}}   

\def\+#1{^{\{#1\}}}

\def\beq{\begin{equation}}
\def\eeq{\end{equation}}
\def\be{\begin{equation*}}
\def\ee{\end{equation*}}

\nc{\bea}{\begin{eqnarray*}}
\nc{\eea}{\end{eqnarray*}}
\nc{\bean}{\begin{eqnarray}}
\nc{\eean}{\end{eqnarray}}

\def\g{{\mathfrak g}}

\let\Ga\Gamma

\nc{\Il}{{\mc I_{\bs\la}}}
\nc{\bla}{{\bs\la}}
\nc{\Fla}{\F_\bla}
\nc{\tfl}{{T^*\Fla}}
\nc{\GL}{{GL_n(\C)}}
\nc{\GLC}{{GL_n(\C)\times\C^*}}

\let\sd s 

\def\ddk_#1{\kk_{#1}\<\>\frac\der{\der\<\>\kk_{#1}}}

\def\bul{\mathbin{\raise.2ex\hbox{$\sssize\bullet$}}}
\def\intt{\mathchoice
{\mathop{\raise.2ex\rlap{$\,\,\ssize\backslash$}{\intop}}\nolimits}
{\mathop{\raise.3ex\rlap{$\,\sssize\backslash$}{\intop}}\nolimits}
{\mathop{\raise.1ex\rlap{$\sssize\>\backslash$}{\intop}}\nolimits}
{\mathop{\rlap{$\sssize\<\>\backslash$}{\intop}}\nolimits}}

\let\kk q 
\let\cc c

\let\Ko K

\def\GZ/{Gelfand-Zetlin}
\def\KZ/{{\slshape KZ\/}}
\def\qKZ/{{\slshape qKZ\/}}
\def\XXX/{{\slshape XXX\/}}

\nc{\A}{{\mc A}}

\def\sll{{\frak{sl}}}

\def\Q{{\mathbb Q}}

\nc{\hsl}{\widehat{{\frak{sl}_2}}}

\nc{\BC}{{ \mathbb C}}
\nc{\lra}{\longrightarrow}
\nc{\CO}{{\mathcal{O}}}
\nc{\BZ}{{ \mathbb Z}}
\nc{\hfn}{\hat{\frak{n}}}
\nc\Zs{{\Z/p^s\Z}}
\nc\Zo{{\Zs[z]^0}}
\nc\gr{{\on{gr}}}

\nc\fD{{\frak D}}

\newcommand{\Cf}{\operatorname{Coeff}}

\begin{document}

\hrule width0pt
\vsk->

\title[Dwork--type congruences and $p$-adic \KZ/ connection]
{Dwork--type congruences and $p$-adic \KZ/ connection}

\author[Alexander Varchenko]
{Alexander Varchenko}

\maketitle

\begin{center}
{\it $^{\star}$ Department of Mathematics, University
of North Carolina at Chapel Hill\\ Chapel Hill, NC 27599-3250, USA\/}



\end{center}

\vsk>
{\leftskip3pc \rightskip\leftskip \parindent0pt \Small
{\it Key words\/}:  \KZ/ equations; Dwork--type congruences;  Hasse--Witt matrices.

\vsk.6>
{\it 2020 Mathematics Subject Classification\/}: 11D79 (12H25, 32G34, 33C05, 33E30)
\par}

{\let\thefootnote\relax
\footnotetext{\vsk-.8>\noindent
$^\star\<${\sl E\>-mail}:\enspace anv@email.unc.edu,
supported in part by NSF grant DMS-1954266
}}

\begin{abstract}
We show that the  $p$-adic \KZ/ connection associated with the family of curves
$y^q=(t-z_1)\dots (t-z_{qg+1})$ has an invariant  subbundle of rank $g$, while
the corresponding complex \KZ/ connection has no  nontrivial proper subbundles due to the 
irreducibility  of its monodromy representation.
The construction of the invariant subbundle is based on new Dwork--type  congruences for
associated Hasse--Witt matrices.

\end{abstract}

{\small\tableofcontents\par}

\setcounter{footnote}{0}
\renewcommand{\thefootnote}{\arabic{footnote}}

\section{Introduction}

 The Knizhnik--Zamolodchikov (\KZ/) differential equations  are objects 
of conformal field theory, representation theory, enumerative geometry, see for example \cite{KZ, Dr, EFK, MO, V2}.
The solutions of the \KZ/ equations have the form of multidimensional hypergeometric functions,
see \cite{SV1}. In this paper we discuss the analog of hypergeometric solutions of the KZ equations considered over a $p$-adic field
instead of the field of complex numbers.

\vsk.2>

More precisely, we consider the  \KZ/ equations in the special case, in which the complex hypergeometric
solutions are given by the integrals of the form
\bean
\label{Iq}
I(z_1,\dots,z_{qg+1})=
\int_C\frac{R(t,z_1,\dots, z_{qg+1})\, dt}
{\sqrt[1/q]{(t-z_1)\dots(t-z_{qg+1})}}
\eean
where $q, g$ are positive integer parameters, and  $R(t,z)$ are suitable rational functions.
 
In this case the space of solutions of the \KZ/ equations is a $qg$-dimensional complex vector space.
We also consider the $p$-adic version of the same differential
equations.  We 
assume that $q$ is a prime number (that is a technical assumption) 
and show that the $qg$-dimensional space of
 local solutions of these $p$-adic 
\KZ/ equations has a remarkable $g$-dimensional subspace of
 solutions which can be $p$-adic analytically
continued as a subspace to a large domain $\frak D_{\on{KZ}}^{(m),o}$
in the space where the \KZ/ equations are defined,   see Theorems \ref{thm inv} and \ref{thm rk g} for precise statements.
This $g$-dimensional global subspace of solutions is defined  as the
 uniform $p$-adic limit of a
$g$-dimensional space of polynomial solutions of these \KZ/ equations 
modulo $p^s$ as $s\to\infty$. For $q=2$ and $g=1$
this construction was deduced in \cite{V5} from the classical  B.\,Dwork's paper \cite{Dw}, 
see also \cite{VZ1}. For $q=2$ and any $g$ the corresponding construction was developed in \cite{VZ2}.

\vsk.2>

In \cite{SV2} general \KZ/  equations were considered over the field $\F_p$ and their polynomial solutions
were constructed as $p$-approximations of hypergeometric integrals.
 In the current paper that construction is modified to obtain 
 polynomial solutions modulo $p^s$  of the \KZ/ equations related to the integrals in formula \eqref{Iq}.
The polynomial solutions are vectors of polynomials with integer coefficients. We call them
the  $p^s$-hypergeometric solutions.  
While the complex analytic  integrals in \eqref{Iq} give the whole $qg$-dimensional space of all solutions of the 
complex \KZ/ equations, the $p^s$-hypergeometric  solutions span only a $g$-dimensional subspace.
Then the $p$-adic limit of that
subspace as $s\to\infty$ gives the desired globally defined subspace of solutions.

On other $p$-approximations of hypergeometric periods see \cite{SV2, RV1, RV2, VZ1, VZ2}.

\vsk.2>

In order to prove Theorems \ref{thm inv} and \ref{thm rk g} we develop 
new matrix Dwork--type congruences in Section \ref{sec 2}. 
In Section \ref{sec 3} we show  how our Dwork--type congruences imply the uniform
$p$-adic  convergence of certain sequences of matrices on suitable domains of the space of their parameters.
In Section \ref{sec 4}
we define our \KZ/ equations and construct their complex holomorphic solutions.
In Section \ref{sec 5} we describe the $p^s$-hypergeometric solutions of the same equations.
In Section \ref{sec 6} we formulate and prove the main Theorems \ref{thm inv} and \ref{thm rk g}.

\vsk.2>

This paper may be viewed as a continuation of the paper \cite{VZ2} where the case $q=2$ is developed.

\medskip

The author thanks
 Louis Funar, Toshitake Kohno, Nick Salter, and 
 Wadim Zudilin for useful discussions.
 The author thanks  Max Planck Institute for Mathematics in Bonn for hospitality in May-June of 2022.

\section{Dwork--type congruences}
\label{sec 2}

The Dwork--type congruences were originated by B.\,Dwork in the classical paper \cite{Dw}.
On Dwork--type congruences see for example
 \cite{Dw, Me, MV, Vl, VZ1, VZ2}.

\vsk.2>

In this paper $p$ is an odd prime.  We denote by $\Z_p[w^{\pm1}]$ the ring of
 Laurent polynomials in variables $w$ with coefficients in $\Z_p$.
A congruence $F(w)\equiv G(w)\pmod{p^s}$ for two Laurent polynomials from the ring is understood 
as the divisibility by $p^s$ of all coefficients of $F(w)-\nobreak G(w)$.

\vsk.2>

For a Laurent polynomial $G(w)$ we define $\si(G(w))=G(w^p)$.

\vsk.2>

We denote $x=(t,z)$, where $t=(t_1,\dots,t_r)$ and $z=(z_1,\dots,z_n)$ are two groups of variables.

\subsection{Definition of ghosts}

Let $ {\bf e}=(e_1, \dots, e_{l})$ be a tuple of positive integers and  $\La=(\La_0(x),\La_1(x), \dots, \La_l(x))$ 
a tuple of Laurent polynomials in $\Z_p[x^{\pm1}]$.

Define $V_0(x)=\La_0(x)$. For $s=1,\dots,l$, define 
$V_s(x)$ 
by the recursive formula
\bean
\label{dls+}
&&
\La_0(x)\La_1(x)^{p^{e_1}}\dots \La_s(x)^{p^{e_1+\dots+e_s}}
=V_s(x) + V_{s-1}(x) \La_s(x^{p^{e_1+\dots+e_s}}) +
\\
&&
\notag
\phantom{aaa}
+ V_{s-2}(x) \La_{s-1}(x^{p^{e_1+\dots+e_{s-1}}})
\La_{s}(x^{p^{e_1+\dots+e_{s-1}}})^{p^{e_s}} + \dots +
\\
&&
\notag
\phantom{aaaaaa}
+ V_{0}(x) \La_{1}(x^{p^{e_1}})\La_{2}(x^{p^{e_1}})^{p^{e_2}}\cdots
\La_{s}(x^{p^{e_1}})^{p^{e_2+\dots+e_s}},
\eean
The Laurent polynomials $V_0(x), \dots, V_l(x)  \in \Z_p[x^{\pm1}]$ are called
 the {\it ghosts} associated with the tuples ${\bf e}$ and $\La$.

\vsk.2>

For every $0\leq j\leq s\leq l$, denote
\bea
W_s(x)
&:=&
\La_0(x)\La_1(x)^{p^{e_1}}\cdots \La_s(x)^{p^{e_1+\dots+e_s}},
\\
W_s^{(j)}(x)
&:=&
\La_j(x)\La_{j+1}(x)^{p^{e_{j+1}}}\cdots 
\La_s(x)^{p^{e_{j+1}+\dots +e_s}}.
\eea
Then \eqref{dls+} can be formulates as
\begin{equation}
\label{dls}
W_s(x)
=V_s(x) + \sum_{j=1}^sV_{j-1}(x)W_s^{(j)}(x^{p^{e_{1}+\dots +e_j}}),
\end{equation}
or as 
\begin{equation}
\label{dls-}
W_s(x)
=V_s(x) + \sum_{j=1}^sV_{j-1}(x)\si^{e_{1}+\dots +e_j}(W_s^{(j)}(x)).
\end{equation}

\begin{lem}
\label{lem dl}

For $s=0,1,\dots,l$, we have $V_s(x) \equiv 0 \pmod{p^{s}}$.
\end{lem}

\begin{proof} 
In the proof we use the congruence $F(x^{p})^{p^{i-1}}\equiv F(x)^{p^{i}}\pmod{p^{i}}$ valid for $i> 0$.

For $s=0$ we have $V_0(x)=\La_0(x)$ and no requirements on divisibility.
For $s=1$, we have 
\bea
V_1(x) = \La_0(x)\La_1(x)^{p^{e_1}} -V_0(x)\La_1(x^{p^{e_1}})
= \La_0(x)(\La_1(x)^{p^{e_1}} -\La_1(x^{p^{e_1}})) ,
\eea
and 
\bean
\label{L1}
&&
\\
&&
\notag
\La_1(x^{ p^{e_1}}) \stackrel{\pmod{p^{}}}{\equiv} \La_1(x^{p^{e_1-1}})^p 
\stackrel{\pmod{p^{2}}}{\equiv} \La_1(x^{p^{e_1-2}})^{p^2} 
\stackrel{\pmod{p^{3}}}{\equiv}    \dots 
\stackrel{\pmod{p^{e_1}}}{\equiv} \La_1(x)^{p^{e_1}}.
\eean
This proves the lemma for $s=1$.

For $s>1$ the proof is by induction on $s$. Assume that the lemma is proved for all $j<s$. Then
similarly to \eqref{L1} we obtain
$\La_s(x^{p^{e_1+\dots+e_j}})^{p^{e_{j+1}+\dots+e_s}}{\equiv}
\La_s(x)^{p^{e_{1}+\dots+e_s}} \pmod{p^{1+e_{j+1}+\dots+e_s}}$
and hence
\bea
V_{j-1}(x)\La_s(x^{p^{e_1+\dots+e_j}})^{p^{e_{j+1}+\dots+e_s}} 
&\equiv&
V_{j-1}(x) \La_s(x)^{p^{e_1+\dots+e_s}} \pmod{p^{j+e_{j+1}+\dots+e_s}}
\\
&\equiv&
V_{j-1}(x) \La_s(x)^{p^{e_1+\dots+e_s}} \pmod{p^{s}}
\eea
since $e_i\geq 1$ for all $i$. Then we  deduce modulo $p^{s}$:
\begin{align*}
V_s(x)
&=W_{s-1}(x)\La_s(x)^{p^{e_{1}+\dots +e_s}}
-\sum_{j=1}^{s-1}V_{j-1}(x) W_{s-1}^{(j)}(x^{p^{e_{1}+\dots +e_j}})
\La_s(x^{p^{e_{1}+\dots +e_j}})^{p^{e_{j+1}+\dots +e_s}}
\\
&
-V_{s-1}(x)\La_s(x^{p^{e_{1}+\dots +e_s}})
\equiv
\\ &
\equiv\bigg(W_{s-1}(x)
-\sum_{j=1}^{s-1}V_{j-1}(x)W_{s-1}^{(j)}(x^{p^{e_{1}+\dots +e_j}})
-V_{s-1}(x)\bigg)\La_s(x)^{p^{e_{1}+\dots +e_s}}
=0,
\end{align*}
obtaining the required statement.
\end{proof}

For a Laurent polynomial $F(t,z)$ in $t,z$, let $N(F)\subset \R^r$ be the Newton polytope of $F(t,z)$ with respect to 
the $t$ \emph{variables only}.

\begin{lem}
\label{lem 1.2}
For $s=0,1,\dots,l$, we have
\bea
N(V_s) \subset N(\La_0)+ p^{e_1}N(\La_1) + \dots+p^{e_1+\dots+e_s}N(\La_s)\,.
\eea

\end{lem}

\begin{proof}
This follows from \eqref{dls} by induction on $s$.
\end{proof}

\smallskip

\subsection{Convex polytopes}

Let $\Dl=(\Dl_0, \dots,  \Dl_l)$ be a tuple of nonempty finite subsets of $\Z^r$  of the same size
$\# \Dl_j=g$ for some positive integer $g$.

\begin{defn}
\label{defN}
A tuple $(N_0,N_1,\dots,N_l)$ of convex polytopes in $\R^r$
is called $(\Dl,{\bf e})$-\emph{admissible} if for any
$0\le i\le j < l$ we have
\bean
\label{def ad}
\phantom{aaaaaa}
\big(\Dl_i + N_i+ p^{e_{i+1}}N_{i+1} + \dots+p^{e_{i+1}+\dots +e_{j}}N_j\big)
\cap p^{e_{i+1}+\dots +e_{j+1}}\Z^r  \subset p^{e_{i+1}+\dots +e_{j+1}}\Dl_{j+1}\,.
\eean

\end{defn}

Notice that any sub-tuple $(N_i,N_{i+1},\dots,N_j)$ of a $(\Dl,{\bf e})$-admissible
tuple $(N_0,N_1,\dots,N_l)$ is $(\Dl',{\bf e}')$-admissible where
$\Dl'=(\Dl_i,\dots,\Dl_j)$ and ${\bf e}'=(e_{i+1},\dots,e_j)$.

\begin{defn}
\label{defn}

A tuple $(\La_0(t,z),\La_1(t,z),\dots,\La_l(t,z))$ of Laurent polynomials 
is called $(\Dl,{\bf e})$-\emph{ad\-missible}
if the tuple 
$\big(N(\La_0), N(\La_1), \dots, N(\La_l)\big)$ is $(\Dl,{\bf e})$-admissible.

\end{defn}

\begin{example}
Let $r=1$, $n=13$, ${\bf e}=(2,2,\dots,2)$,  $\Ga=\{1,2,3,4\}\subset \Z$, 
$\Dl=(\Ga,\Ga, \dots,\Ga)$, 
\linebreak
$N=[0,13(p^2-1)/3]\subset \R$,
$F(t_1,z) =\prod_{i=1}^{13}(t_1-z_i)^{(p^2-1)/3}$. Then
the tuple $(N,N, \dots, N)$ of intervals in $\R$ and
the tuple of polynomials $(F(t_1,z), F(t_1,z), \dots, F(t_1,z))$
are $(\Dl,{\bf e})$-admissible.

\end{example}

\smallskip

\subsection{Hasse--Witt matrices}

For $v\in\Z^r$ denote by $\Cf_v F(t,z)$ the coefficient of $t^v$ in the Laurent polynomial $F(t,z)$.
This is a Laurent polynomial in $z$.

\vsk.2>

Given $m\ge 1$ and finite subsets $\Dl',\Dl''\subset \Z^r$,
define  the \emph{Hasse--Witt matrix} of the Laurent polynomial $F(t,z)$
by the formula
\bean
\label{Cuv+}
A(m, \Dl',\Dl'', F(t,z))
:=
\big( \Cf_{p^mv-u} F(t,z)\big)_{u\in\Dl', v\in\Dl''}\,.
\eean

\begin{lem}
\label{lem 2.5}

Let $\La$ be a $(\Dl,{\bf e})$-admissible tuple of Laurent polynomials in 
\linebreak
$\Z_p[x^{\pm1}]=\Z_p[t^{\pm1}, z^{\pm1}]$. 
Then for $0\le s\le l$ we have
\begin{alignat*}{2}
\textup{(i)} &\;\;
A(e_{1}+\dots+e_{s+1}, \Dl_0,\Dl_{s+1}, V_s) \equiv 0 \pmod{p^s};
\\
\textup{(ii)} &\;\; 
A\big(e_{1}+\dots+e_{s+1}, \Dl_0,\Dl_{s+1},  W_s\big)
=
\\
&
=
A\big(e_1,\Dl_0,\Dl_{1}, V_0) \cdot  \si^{e_1}\big(A\big(e_{2}+\dots+e_{s+1},\Dl_1,\Dl_{s+1},  W_s^{(1)}\big)\big)
+
\\
&
+A\big(e_1+e_2,\Dl_0,\Dl_{2}, V_1) \cdot \si^{e_{1}+e_{2}}\big(A\big(e_{3}+\dots+e_{s+1}, \Dl_2,\Dl_{s+1}, W_s^{(2)}\big)\big)
+\dots 
+ 
\\ 
&
+A\big(e_{1}+\dots+e_{s},\Dl_0,\Dl_{s}, V_{s-1}) \cdot \si^{e_{1}+\dots+e_{s}}\big(A\big(e_{s+1},\Dl_s,\Dl_{s+1},  W_s^{(s)}\big)\big)
+
\\
&
+ A\big(e_{1}+\dots+e_{s+1},\Dl_0,\Dl_{s+1}, V_s).
\end{alignat*}

\end{lem}

Notice that all these matrices are $g\times g$-matrices.

\begin{proof}  

Part (i) follows from Lemma \ref{lem dl}. To prove (ii) consider the identity
\begin{align}
\label{jth}
&
\La_0(t,z)\La_1(t,z)^{p^{e_1}}\dots \La_s(t,z)^{p^{e_1+\dots+e_s}}
=
\sum_{j=1}^sV_{j-1}(t,z)\La_j(t^{p^{e_1+\dots+e_j}},z^{p^{e_1+\dots+e_j}}) \times
\\ &
\notag
\qquad
\times \La_{j+1}(t^{p^{e_1+\dots+e_j}},z^{p^{e_1+\dots+e_j}})^{p^{e_{j+1}}}\dots
\La_s(t^{p^{e_1+\dots+e_j}},z^{p^{e_1+\dots+e_j}})^{p^{e_{j+1}+\dots+e_s}}+ V_s(t,z),
\end{align}
which is nothing else but \eqref{dls+}.
Let $u\in\Dl_0,v\in\Dl_{s+1}$. In order to calculate the coefficient of $t^{p^{e_{1}+\dots+e_{s+1}}v-u}$
in the $j$-th summand on the right-hand side of \eqref{jth},
we look for all pairs of vectors $w\in N(V_{j-1})$ and $y \in N(\La_j(t,z)\dots\La_s(t,z)^{p^{e_{j+1}+\dots+e_{s+1}}})$
such that
\bea
w+p^{e_1+\dots+e_j}y = p^{e_1+\dots+e_{s+1}}v-u.
\eea
Hence  $u+w\in p^{e_1+\dots+e_j} \Z^r$. On the other hand, it follows from Lemma \ref{lem 1.2}
that $w \in N(\La_0)+ p^{e_1}N(\La_1) + \dots+p^{e_1+\dots+e_{j-1}}N(\La_{j-1})$, so that
\bea
u+w\in \Dl_0 + N(\La_0)+ pN(\La_1) + \dots+p^{e_1+\dots+e_{j-1}}N(\La_{j-1}).
\eea
From the $(\Dl,{\bf e})$-admissibility we deduce that
$u+w = p^{e_1+\dots+e_{j}} \dl$ for some $\dl\in\Dl_j$, thus
$w=p^{e_1+\dots+e_j} \dl -\nobreak u$, \  $y=p^{e_{j+1}+\dots+e_{s+1}}v-\dl$ and
\begin{align*}
&
\Cf_{p^{e_1+\dots+e_{s+1}}v-u}
\big(V_{j-1}(t,z)\La_j(t^{p^{e_1+\dots+e_{j}}},
z^{p^{e_1+\dots+e_{j}}})\dots\La_s(t^{p^{e_1+\dots+e_j}},z^{p^{e_1+\dots+e_j}})^{p^{e_{j+1}+\dots+e_s}}
\big)
=
\\ &
\phantom{aaa}
=\sum_{\dl\in\Dl_j} \Cf_{p^{e_1+\dots+e_{j}}\dl-u}(V_{j-1}(t,z))\, \cdot
\\
&
\phantom{aaaaaa}
\cdot\,
\si^{e_1+\dots+e_j}\big(\Cf_{p^{e_{j+1}+\dots+e_{s+1}}v-\dl}
\big(\La_j(t,z)\La_{j+1}(t,z)^{p^{e_{j+1}} }\dots
\La_s(t,z)^{p^{e_{j+1}+\dots+e_{s}}}\big)\big).
\end{align*}
This proves  (ii).
\end{proof}

\subsection{Congruences}

The next results discuss congruences of the type
\\
$F_1(z)F_2(z)^{-1}\equiv G_1(z)G_2(z)^{-1}\pmod{p^s}$, where
 $F_1,F_2,G_1,G_2$ are $g\times g$ matrices whose entries are Laurent polynomials in $z$.
We consider such congruences when 
the determinants $\det F_2(z)$ and $\det G_2(z)$
are  Laurent polynomials  both nonzero  modulo~$p$.
Using Cramer's rule we write
  the entries of the inverse matrix $F_2(z)^{-1}$ in the 
 form $f_{ij}(z)/\det F_2(z)$ for $f_{ij}(z)\in\Z_p[z^{\pm1}]$ and do a similar computation for $G_2(z)$.
This presents the congruence $F_1(z)F_2(z)^{-1}\equiv G_1(z)G_2(z)^{-1}\pmod{p^s}$ 
in the form 
\bean
\label{ff=gg}
\frac1{\det F_2(z)}\cdot F(z)\ \equiv\  \frac1{\det G_2(z)}\cdot G(z) \pmod{p^s}
\eean 
for some $g\times g$ matrices $F(z), G(z)$ with entries in $\Z_p[z^{\pm1}]$,
while \eqref{ff=gg} is nothing else but the congruence 
$F(z)\cdot \det G_2(z)\equiv  G(z)\cdot \det F_2(z) \pmod{p^s}$.

\begin{thm}
\label{thm 1.6}
Let $(\La_0(t,z),\La_1(t,z),\dots,\La_l(t,z))$ be a $(\Dl,{\bf e})$-admissible
 tuple of Laurent polynomials in $\Z_p[x^{\pm1}]=\Z_p[t^{\pm1}, z^{\pm1}]$. 
\begin{enumerate}
\item[\textup{(i)}] For $0\leq s\leq l$ we have
\begin{align}
\notag
&
A\big(e_1+\dots+e_{s+1}, \Dl_0,\Dl_{s+1}, \La_0(x)\La_1(x)^{p^{e_1}}\cdots \La_s(x)^{p^{e_1+\dots+e_s}}
\big)\equiv
\\
\notag
&
\equiv A\big(e_1,\Dl_0,\Dl_1, \La_0(x)\big) \cdot
\si^{e_1}\big(A\big(e_2,\Dl_1,\Dl_2, \La_1(x)\big)\big) 
\cdots \si^{e_1+\dots+e_s}\big(A\big(e_{s+1},\Dl_s,\Dl_{s+1}, \La_s(x)\big)\big)
\end{align}
modulo $p$.

\item[\textup{(ii)}] Assume that the determinants of  the matrices $ A\big(e_{i+1},\Dl_i,\Dl_{i+1}, \La_i(t,z)\big)$, $i=0,1,\dots,l$, 
are Laurent polynomials all nonzero modulo~$p$.
Then for $1\leq s \leq l$ 
the determinant of the matrix 
$A\big(e_2+\dots+e_{s+1},\Dl_1, \Dl_{s+1}, 
\La_1(x)\La_2(x)^{p^{e_2}}\cdots \La_s(x)^{p^{e_2+\dots+e_{s}}} \big)$
is a Laurent polynomial nonzero modulo~$p$ and 
we have modulo $p^s$\,\textup:
\begin{align}
\label{s cong}
&
A\big(e_1+\dots+e_{s+1}, \Dl_0,\Dl_{s+1}, \La_0(x)\La_1(x)^{p^{e_1}}\cdots \La_s(x)^{p^{e_1+\dots+e_s}}\big)
\cdot
\\
\notag
&
\cdot 
\si^{e_1}\big(
A\big(e_2+\dots+e_{s+1},\Dl_1, \Dl_{s+1}, 
\La_1(x)\La_2(x)^{p^{e_2}}\cdots \La_s(x)^{p^{e_2+\dots+e_{s}}} \big)\big)^{-1}
\equiv
\\
\notag
&
\equiv
A\big(e_1+\dots+e_{s},\Dl_0,\Dl_{s}, \La_0(x)\La_1(x)^{p^{e_1}}\cdots \La_{s-1}(x)^{p^{e_1+\dots+e_{s-1}}}\big)
\cdot
\\
\notag
&
\cdot
\si^{e_1}\big(A\big(e_2+\dots+e_s,\Dl_1,\Dl_s, \La_1(x)\La_2(x)^{p^{e_2}}\cdots \La_{s-1}(x)^{p^{e_2+\dots+e_{s-1}}}\big)\big)^{-1},
\end{align}
where in this congruence  for $s=1$ we 
understand the second factor on the right-hand side  as 
 the  $g\times g$ identity matrix, see formula \eqref{s=1} below.

\end{enumerate}

\end{thm}

\begin{proof}
By Lemma \ref{lem 2.5} we have
\begin{align*}
&
A\big(e_1+\dots+e_{s+1},\Dl_0,\Dl_{s+1}, \La_0(x)\La_1(x)^{p^{e_1}}\dots \La_s(x)^{p^{e_1+\dots+e_{s}}}\big)
\equiv
\\
&
\equiv
A\big(e_1,\Dl_0,\Dl_{1}, \La_0(x)) \cdot  \si^{e_1}\big(A\big(e_{2}+\dots+e_{s+1},\Dl_1,\Dl_{s+1},  
\La_1(x)\La_2(x)^{p^{e_2}}\dots \La_s(x)^{p^{e_2+\dots+e_{s}}}\big)\big)
\end{align*}
modulo $p$.
Iteration gives part (i) of the theorem.

\vsk.2>
If the determinants of  the matrices $ A\big(e_{i+1},\Dl_i,\Dl_{i+1}, \La_i(t,z)\big)$, $i=0,1,\dots,l$, 
are Laurent polynomials all nonzero modulo~$p$, then part (i) implies that
the determinant
\bea
&
\det A\big(e_2+\dots+e_{s+1},\Dl_1, \Dl_{s+1}, 
\La_1(x)\La_2(x)^{p^{e_2}}\cdots \La_s(x)^{p^{e_2+\dots+e_{s}}} \big)
\equiv 
\\
&
\equiv
\prod_{j=1}^s \det \si^{e_2+\dots+e_{j}}\big(A\big(e_{j+1},\Dl_j,\Dl_{j+1}, \La_j(t,z)\big)\big) \pmod{p},
\eea
is a Laurent polynomial nonzero modulo $p$.  
This proves the first statement of part (ii) of the theorem and allows us to consider the inverse matrices 
in the congruence of part (ii).

\vsk.2>

We prove part (ii) by induction on $s$. For $s=1$, congruence \eqref{s cong}
takes the form
\bean
\label{s=1}
&
\\
\notag
&
A\big(e_1+e_2,\Dl_0,\Dl_2,\La_0(x)\La_1(x)^{p^{e_1}}\big)\cdot 
\si^{e_1}\big(A\big(e_2,\Dl_1,\Dl_2,\La_1(x)\big)\big)^{-1}
\equiv
A\big(e_1,\Dl_2,\Dl_1,\La_0(x)\big)
\eean
modulo $p$. This congruence  follows from part (i).

\vsk.2>

For $1< s < l$ 
we substitute the expressions for 
$A\big(e_1+\dots+e_{s+1}, \Dl_0,\Dl_{s+1}, \La_0(x)\La_1(x)^{p^{e_1}}\cdots $
$\cdots\La_s(x)^{p^{e_1+\dots+e_s}}\big)$
and
$A\big(e_1+\dots+e_{s},\Dl_0,\Dl_{s}, \La_0(x)\La_1(x)^{p^{e_1}}\cdots \La_{s-1}(x)^{p^{e_1+\dots+e_{s-1}}}\big)$
from part (ii) of Lemma~\ref{lem 2.5} into the two sides of the desired congruence:
\bean
\label{con1}
&
A\big(\sum_{a=1}^{s+1}e_a, \Dl_0,\Dl_{s+1}, \La_0(x)\La_1(x)^{p^{e_1}}\cdots \La_s(x)^{p^{e_1+\dots+e_s}}\big)
\cdot
\\
\notag
&
\cdot 
\si^{e_1}\big(A\big(\sum_{a=2}^{s+1}e_a,\Dl_1, \Dl_{s+1},
\La_1(x)\La_2(x)^{p^{e_2}}\cdots \La_s(x)^{p^{e_2+\dots+e_{s}}} \big)\big)^{-1} = A\big(e_1,\Dl_0,\Dl_{1}, V_0) +
\\
\notag
&
+ \sum_{j=2}^s
A\big(\sum_{a=1}^{j}e_a ,\Dl_0,\Dl_{j+1}, V_{j-1}) \cdot 
\si^{\sum_{a=1}^{j}e_a}\big(A\big(\sum_{a=j+1}^{s+1}e_{a},\Dl_j,\Dl_{s+1}, 
 W_s^{(j)}\big)\big)
\cdot
\\
&
\notag
\cdot
\,
\si^{e_1}\big(A\big(\sum_{a=2}^{s+1}e_a,\Dl_1, \Dl_{s+1}, W_s^{(1)} \big)\big)^{-1}
+ 
\\
\notag
&
+ \, A\big(\sum_{a=1}^{s+1}e_a, \Dl_0,\Dl_{s+1},V_s  \big)\cdot
\si^{e_1}\big(A\big(\sum_{a=2}^{s+1}e_a,\Dl_1, \Dl_{s+1}, W_s^{(1)} \big)\big)^{-1}
\eean
and
\bean
\label{con2}
&
A\big(\sum_{a=1}^{s}e_a, \Dl_0,\Dl_{s}, \La_0(x)\La_1(x)^{p^{e_1}}\cdots \La_{s-1}(x)^{p^{e_1+\dots+e_{s-1}}}\big)
\cdot
\\
\notag
&
\cdot 
\si^{e_1}\big(A\big(\sum_{a=2}^{s}e_a,\Dl_1, \Dl_{s},
\La_1(x)\La_2(x)^{p^{e_2}}\cdots \La_{s-1}(x)^{p^{e_2+\dots+e_{s-1}}} \big)\big)^{-1} = A\big(e_1,\Dl_0,\Dl_{1}, V_0) +
\\
\notag
&
+ \sum_{j=2}^s
A\big(\sum_{a=1}^{j}e_a ,\Dl_0,\Dl_{j+1}, V_{j-1}) \cdot \si^{\sum_{a=1}^{j}e_a}\big(A\big(\sum_{a=j+1}^{s}e_{a},\Dl_j,\Dl_{s}, 
 W_{s-1}^{(j)}\big)\big)
\cdot
\\
&
\notag
\cdot
\,
\si^{e_1}\big(A\big(\sum_{a=2}^{s}e_a,\Dl_1, \Dl_{s}, W_{s-1}^{(1)} \big)\big)^{-1}.
\eean
Since we want to compare these two expressions modulo $p^s$,
the last term in \eqref{con1} containing $V_s \equiv 0 \pmod{p^s}$ can be ignored.

Given $j = 2, \dots , s$, we use the inductive hypothesis as
follows:
\bea
&
A\big(\sum_{a=i+1}^{s+1}e_a, \Dl_{i},\Dl_{s+1}, W^{(i)}_{s}\big)
\cdot
\si^{e_{i+1}}\big(
A\big(\sum_{a=i+2}^{s+1}e_a,\Dl_{i+1}, \Dl_{s+1},W^{(i+1)}_{s} \big)\big)^{-1}
\equiv
\\
\notag
&
\equiv
A\big(\sum_{a=i+1}^{s}e_a, \Dl_{i},\Dl_{s}, W^{(i)}_{s-1}\big)
\cdot
\si^{e_{i+1}}\big(
A\big(\sum_{a=i+2}^{s}e_a,\Dl_{i+1}, \Dl_{s},W^{(i+1)}_{s-1} \big)\big)^{-1}
\pmod{p^{s-i}}
\eea
for $i=1,\dots,j-1$. Applying $\sigma^{\sum_{a=1}^ie_a}$ to the $i$-th congruence
 and multiplying them out lead to telescoping products on both sides:
\bea
&
\si^{e_1}\big(A\big(\sum_{a=2}^{s+1}e_a,\Dl_1, \Dl_{s+1}, W_s^{(1)} \big)\big)
\cdot
\si^{\sum_{a=1}^{j}e_a}\big(A\big(\sum_{a=j+1}^{s+1}e_{a},\Dl_j,\Dl_{s+1}, 
 W_s^{(j)}\big)\big)^{-1}\equiv
\\
&
\equiv 
\si^{e_1}\big(A\big(\sum_{a=2}^{s}e_a,\Dl_1, \Dl_{s}, W_{s-1}^{(1)} \big)\big)
\cdot \si^{\sum_{a=1}^{j}e_a}\big(A\big(\sum_{a=j+1}^{s}e_{a},\Dl_j,\Dl_{s}, 
 W_{s-1}^{(j)}\big)\big)^{-1} 
 \eea
modulo $p^{s-j+1}$.
 By our assumptions these four matrices are invertible.
Therefore, we can invert them to obtain the congruence
\bean
\label{17}
&
\phantom{aaa}
\si^{\sum_{a=1}^{j}e_a}\big(A\big(\sum_{a=j+1}^{s+1}e_{a},\Dl_j,\Dl_{s+1}, W_s^{(j)}\big)\big)
 \cdot
 \si^{e_1}\big(A\big(\sum_{a=2}^{s+1}e_a,\Dl_1, \Dl_{s+1}, W_s^{(1)} \big)\big)^{-1}
 \equiv
\\
&
\notag
\phantom{aaa}
\equiv 
\si^{\sum_{a=1}^{j}e_a}\big(A\big(\sum_{a=j+1}^{s}e_{a},\Dl_j,\Dl_{s}, 
 W_{s-1}^{(j)}\big)\big)\cdot \si^{e_1}\big(A\big(\sum_{a=2}^{s}e_a,\Dl_1, \Dl_{s}, W_{s-1}^{(1)} \big)\big)^{-1}
\eean
modulo $p^{s-j+1}$.
Since $V_{j-1}\equiv 0\pmod{p^{j-1}}$, we obtain 
the congruence
\bea
&
A\big(\sum_{a=1}^{j}e_a ,\Dl_0,\Dl_{j+1}, V_{j-1}) \cdot
\\
&
\cdot\,
\si^{\sum_{a=1}^{j}e_a}\big(A\big(\sum_{a=j+1}^{s+1}e_{a},\Dl_j,\Dl_{s+1}, W_s^{(j)}\big)\big)
 \cdot
 \si^{e_1}\big(A\big(\sum_{a=2}^{s+1}e_a,\Dl_1, \Dl_{s+1}, W_s^{(1)} \big)\big)^{-1}
 \equiv
\\
&
\equiv 
A\big(\sum_{a=1}^{j}e_a ,\Dl_0,\Dl_{j+1}, V_{j-1}) \cdot
\\
&
\cdot\,
\si^{\sum_{a=1}^{j}e_a}\big(A\big(\sum_{a=j+1}^{s}e_{a},\Dl_j,\Dl_{s}, 
 W_{s-1}^{(j)}\big)\big)\cdot \si^{e_1}\big(A\big(\sum_{a=2}^{s}e_a,\Dl_1, \Dl_{s}, W_{s-1}^{(1)} \big)\big)^{-1}
\eea
modulo $p^{s}$. This shows that the $j$-th summands in \eqref{con1} and \eqref{con2} are congruent modulo $p^s$.
The theorem is proved.
\end{proof}

\begin{cor}

Under the assumptions of part \textup{(ii)} of Theorem \textup{\ref{thm 1.6}}
for $1\leq s \leq l$ 
we have\,\textup:
\begin{align*}
&
\det 
A\big(e_1+\dots+e_{s+1}, \Dl_0,\Dl_{s+1}, \La_0(x)\La_1(x)^{p^{e_1}}\cdots \La_s(x)^{p^{e_1+\dots+e_s}}\big)
\cdot
\\
\notag
&
\cdot 
\det
\si^{e_1}\big(A\big(e_2+\dots+e_s,\Dl_1,\Dl_s, \La_1(x)\La_2(x)^{p^{e_2}}\cdots \La_{s-1}(x)^{p^{e_2+\dots+e_{s-1}}}\big)\big)
\equiv
\\
\notag
&
\equiv
\det A\big(e_1+\dots+e_{s},\Dl_0,\Dl_{s}, \La_0(x)\La_1(x)^{p^{e_1}}\cdots \La_{s-1}(x)^{p^{e_1+\dots+e_{s-1}}}\big)
\cdot
\\
\notag
&
\cdot
\det \si^{e_1}\big(
A\big(e_2+\dots+e_{s+1},\Dl_1, \Dl_{s+1}, 
\La_1(x)\La_2(x)^{p^{e_2}}\cdots \La_s(x)^{p^{e_2+\dots+e_{s}}} \big)\big)
\end{align*}
modulo $p^s$.

\end{cor}

\subsection{Derivations}

Recall that $z=(z_1,\dots,z_n)$. Denote
\bea
D_v=\frac{\der}{\der z_v}, \quad v=1,\dots,n.
\eea
Let $F_1(z),F_2(z),G_1(z),G_2(z) \in \Z_p[z^{\pm1}]$  and $\ell\geq 1$.
If
\bea
D_v(F_1(z))\cdot F_2(z)\equiv  D_v(G_1(z))\cdot G_2(z)\pmod{p^s}\,,
\eea
then
\begin{align}
\label{1.5}
&
D_v(\si^\ell(F_1(z)))\cdot\si^\ell(F_2(z)) - D_v(\si^\ell(G_1(z)))\cdot\si^\ell(G_2(z))
=
\\ &\qquad
= D_v(F_1(z^{p^\ell}))\cdot F_2(z^{p^\ell}) - D_v(G_1(z^{p^\ell}))\cdot G_2(z^{p^\ell})
=
\notag
\\ &\qquad
= p^\ell z_v^{p^\ell-1}\big(D_v(F_1(z))\cdot F_2(z) - D_v(G_1(z))\cdot G_2(z)\big)\big|_{z\to z^{p^\ell}}
\equiv
\notag
\\ &\qquad
\equiv 0 \pmod{p^{s+\ell}}.
\notag
\end{align}

\vsk.2>

\begin{thm}
\label{thm der}
Let $(\La_0(t,z),\La_1(t,z), \dots, \La_l(t,z))$ be a 
$(\Dl,{\bf e})$-admissible tuple of Laurent polynomials in $\Z_p[x^{\pm1}]=\Z_p[t^{\pm1},z^{\pm1}]$. 
Let $D=D_v$ for some $v=1,\dots,n$.
Then under the assumptions  of part \textup{(ii)} of Theorem \textup{\ref{thm 1.6}} we have
\bean
\label{Der}
&
D\big(\si^\ell\big(A\big(\sum_{a=1}^{s+1}e_a, \Dl_0,\Dl_{s+1}, W_s\big)\big)\big)
\cdot \si^\ell\big(A\big(\sum_{a=1}^{s+1}e_a, \Dl_0,\Dl_{s+1}, W_s\big)\big)^{-1} \equiv
\\
&
\notag
\equiv
D\big(\si^\ell\big(A\big(\sum_{a=1}^{s}e_a, \Dl_0,\Dl_{s}, W_{s-1}\big)\big)\big)
\cdot \si^\ell\big(A\big(\sum_{a=1}^{s}e_a, \Dl_0,\Dl_{s}, W_{s-1}\big)\big)^{-1}
\pmod{p^{s+\ell}}
\eean
for   $1\leq s \leq l$  and  $0\leq \ell$.

\end{thm}

\begin{proof}
Notice that it is sufficient to establish the congruences \eqref{Der} for $\ell=0$,
 as the general $\ell$ case follows from \eqref{1.5}. So, we assume that $\ell=0$ and 
proceed by induction on $s\ge0$. For $s=0$  the statement is trivially true.

Using part (ii) of Lemma \ref{lem 2.5} we can write
\bean
\label{con3}
&
D\big(A\big(\sum_{a=1}^{s+1}e_a, \Dl_0,\Dl_{s+1}, W_s\big)\big)
\cdot A\big(\sum_{a=1}^{s+1}e_a, \Dl_0,\Dl_{s+1}, W_s\big)^{-1} 
=
\\ 
&
\notag
=
\sum_{j=1}^{s+1}
D\big(A\big(\sum_{a=1}^{j}e_a ,\Dl_0,\Dl_{j+1}, V_{j-1}\big)\big) \cdot 
\si^{\sum_{a=1}^{j}e_a}\big(A\big(\sum_{a=j+1}^{s+1}e_{a},\Dl_j,\Dl_{s+1}, 
 W_s^{(j)}\big)\big)
\cdot
\\
&
\notag
\cdot
\,
A\big(\sum_{a=1}^{s+1}e_a, \Dl_0,\Dl_{s+1}, W_s\big)^{-1} +
\\ 
&
\notag
+
\sum_{j=1}^{s+1}
A\big(\sum_{a=1}^{j}e_a ,\Dl_0,\Dl_{j+1}, V_{j-1}) \cdot 
D\big(\si^{\sum_{a=1}^{j}e_a}\big(A\big(\sum_{a=j+1}^{s+1}e_{a},\Dl_j,\Dl_{s+1}, 
 W_s^{(j)}\big)\big)\big)
\cdot
\\
&
\notag
\cdot
\,
A\big(\sum_{a=1}^{s+1}e_a, \Dl_0,\Dl_{s+1}, W_s\big)^{-1}
\eean
and
\bean
\label{con4}
&
D\big(A\big(\sum_{a=1}^{s}e_a, \Dl_0,\Dl_{s}, W_{s-1}\big)\big)
\cdot 
A\big(\sum_{a=1}^{s}e_a, \Dl_0,\Dl_{s}, W_{s-1}\big)^{-1}=
\\
\notag
&
=
\sum_{j=1}^s
D\big(A\big(\sum_{a=1}^{j}e_a ,\Dl_0,\Dl_{j+1}, V_{j-1}\big)\big)
 \cdot \si^{\sum_{a=1}^{j}e_a}\big(A\big(\sum_{a=j+1}^{s}e_{a},\Dl_j,\Dl_{s}, 
 W_{s-1}^{(j)}\big)\big)
\cdot
\\
&
\notag
\cdot
\,
A\big(\sum_{a=1}^{s}e_a, \Dl_0,\Dl_{s}, W_{s-1}\big)^{-1} +
\\
\notag
&
+ \sum_{j=1}^s
A\big(\sum_{a=1}^{j}e_a ,\Dl_0,\Dl_{j+1}, V_{j-1}) \cdot
D\big( \si^{\sum_{a=1}^{j}e_a}\big(A\big(\sum_{a=j+1}^{s}e_{a},\Dl_j,\Dl_{s}, 
 W_{s-1}^{(j)}\big) \big)\big)
\cdot
\\
&
\notag
\cdot
\,
A\big(\sum_{a=1}^{s}e_a, \Dl_0,\Dl_{s}, W_{s-1}\big)^{-1}.
\eean

The summands corresponding to $j=s+1$ in \eqref{con3}  vanish modulo $p^s$
and can be ignored since $V_s \equiv 0\pmod{p^s}$.

For the same reason 
\bean
\label{saj}
&
D\big(A\big(\sum_{a=1}^{j}e_a ,\Dl_0,\Dl_{j+1}, V_{j-1}\big)\big)\equiv0\pmod{p^{j-1}}.
\eean

We also have
\bean
\label{17a}
&
\phantom{aaa}
\si^{\sum_{a=1}^{j}e_a}\big(A\big(\sum_{a=j+1}^{s+1}e_{a},\Dl_j,\Dl_{s+1}, W_s^{(j)}\big)\big)
 \cdot
A\big(\sum_{a=1}^{s+1}e_a,\Dl_0, \Dl_{s+1}, W_s \big)^{-1}
 \equiv
\\
&
\notag
\phantom{aaa}
\equiv 
\si^{\sum_{a=1}^{j}e_a}\big(A\big(\sum_{a=j+1}^{s}e_{a},\Dl_j,\Dl_{s}, 
 W_{s-1}^{(j)}\big)\big)\cdot A\big(\sum_{a=1}^{s}e_a,\Dl_0, \Dl_{s}, W_{s-1} \big)^{-1}
 \pmod{p^{s-j+1}}.
\eean
This follows from \eqref{17}, in which we take $j+1$ and $s+1$ for $j$ and $s$ and use 
$W_s$ instead of $W_{s+1}^{(1)}$.

Multiplying congruences \eqref{saj} and \eqref{17a} we get
\bean
\label{17b}
&
D\big(A\big(\sum_{a=1}^{j}e_a ,\Dl_0,\Dl_{j+1}, V_{j-1}\big)\big)\cdot
\si^{\sum_{a=1}^{j}e_a}\big(A\big(\sum_{a=j+1}^{s+1}e_{a},\Dl_j,\Dl_{s+1}, W_s^{(j)}\big)\big)
 \cdot
 \\
 &
 \notag
\cdot\,A\big(\sum_{a=1}^{s+1}e_a,\Dl_0, \Dl_{s+1}, W_s \big)^{-1}
 \equiv
\\
&
\notag
\equiv 
D\big(A\big(\sum_{a=1}^{j}e_a ,\Dl_0,\Dl_{j+1}, V_{j-1}\big)\big)
\cdot \si^{\sum_{a=1}^{j}e_a}\big(A\big(\sum_{a=j+1}^{s}e_{a},\Dl_j,\Dl_{s}, 
 W_{s-1}^{(j)}\big)\big)\cdot
 \\
 \notag
 &
 \cdot\,
  A\big(\sum_{a=1}^{s}e_a,\Dl_0, \Dl_{s}, W_{s-1} \big)^{-1} \pmod{p^s}.
\eean
Congruence \eqref{17b} implies that the first sum in \eqref{con3} is congruent to the first sum in \eqref{con4} modulo $p^s$.

To match the second sums we recall the inductive hypothesis in the form
\bean
\label{18.1}
&
\\
&
\notag
D\big(\si^{\sum_{a=1}^{j}e_a}
\big(A\big(\sum_{a=j+1}^{s+1}e_a, \Dl_j,\Dl_{s+1}, W_s^{(j)}\big)\big)\big) \cdot
\phantom{aaaaaaaaaaaa}
\\
&
\notag
\cdot\,
 \si^{\sum_{a=1}^{j}e_a}\big(A\big(\sum_{a=j+1}^{s+1}e_a, \Dl_j,\Dl_{s+1}, W_s^{(j)}\big)\big)^{-1}
\equiv
\\&
\notag
\equiv\,
D\big(\si^{\sum_{a=1}^{j}e_a}\big(A\big(\sum_{a=j+1}^{s}e_a, \Dl_j,\Dl_{s}, W_{s-1}^{(j)}\big)\big)\big) \cdot
\phantom{aaaaaaaaa}
 \\
 \notag
 &
 \phantom{aaaaaaaaa}
 \cdot\,
 \si^{\sum_{a=1}^{j}e_a}\big(A\big(\sum_{a=j+1}^{s}e_a, \Dl_j,\Dl_{s}, W_{s-1}^{(j)}\big)\big)^{-1}
 \pmod{p^s},
\eean
and notice that both sides in \eqref{18.1} are congruent to zero modulo $\si^{\sum_{a=1}^{j}e_a}$ 
by formula \eqref{1.5}. 
Therefore, multiplying congruences \eqref{18.1} and \eqref{17a}  we obtain
\bea
&
\notag
D\big(\si^{\sum_{a=1}^{j}e_a}
\big(A\big(\sum_{a=j+1}^{s+1}e_a, \Dl_j,\Dl_{s+1}, W_s^{(j)}\big)\big)\big) \cdot
A\big(\sum_{a=1}^{s+1}e_a,\Dl_0, \Dl_{s+1}, W_s \big)^{-1}
\equiv
\\&
\notag
\equiv
D\big(\si^{\sum_{a=1}^{j}e_a}\big(A\big(\sum_{a=j+1}^{s}e_a, \Dl_j,\Dl_{s}, W_{s-1}^{(j)}\big)\big)\big) \cdot
  A\big(\sum_{a=1}^{s}e_a,\Dl_0, \Dl_{s}, W_{s-1} \big)^{-1} \pmod{p^s}.
\eea
Multiplying both sides of this congruence by $A\big(\sum_{a=1}^{j}e_a ,\Dl_0,\Dl_{j+1}, V_{j-1}) $
we conclude that the second sum in \eqref{con3} is congruent to the second sum in \eqref{con4} modulo $p^s$.
The theorem is proved. 
\end{proof}

\vsk.2>
There are similar congruences for higher order derivatives of the matrices
\linebreak
$A\big(\sum_{a=1}^{s+1}e_a, \Dl_0,\Dl_{s+1}, W_s\big)$.
We restrict ourselves with the second order derivatives.

\begin{thm}
\label{thm der2}
Let $(\La_0(t,z),\La_1(t,z), \dots, \La_l(t,z))$ be a $(\Dl, {\bf e})$-admissible tuple of Laurent polynomials in 
$\Z_p[x^{\pm1}]=\Z_p[t^{\pm1},z^{\pm1}]$. 
Then under the assumptions  of part \textup{(ii)} of Theorem \textup{\ref{thm 1.6}} we have
\bean
\label{Der2}
&
D_u\big(D_v\big(A\big(\sum_{a=1}^{s+1}e_a,\Dl_0,\Dl_{s+1},W_{s}\big)\big)\big)
\,\cdot\,
A\big(\sum_{a=1}^{s+1}e_a, \Dl_0,\Dl_{s+1},W_{s}\big)^{-1}
\equiv
\\
 &
\notag
\equiv
D_u\big(D_v\big(A\big(\sum_{a=1}^{s}e_a, \Dl_0,\Dl_{s},W_{s-1}\big)\big)\big)
\cdot
A\big(\sum_{a=1}^{s}e_a, \Dl_0,\Dl_{s},W_{s-1}\big)^{-1}
\pmod{p^{s}}
\eean
for all $1\leq u, v\leq n$ and $0\leq s\leq l$.
\end{thm}

\begin{proof}

Notice that, for an invertible  matrix $F(z)$  and a derivation $D$,
we have $D(F^{-1})=-F^{-1}\,D(F)\,F^{-1}$.

We apply the derivation $D_u$ to congruence \eqref{Der} with $D=D_v$ :
\bea
&
D_u\big(D_v\big(A\big(\sum_{a=1}^{s+1}e_a, \Dl_0,\Dl_{s+1}, W_s\big)\big)\big)
\cdot A\big(\sum_{a=1}^{s+1}e_a, \Dl_0,\Dl_{s+1}, W_s\big)^{-1} 
+
\\
&
+\,D_v\big(A\big(\sum_{a=1}^{s+1}e_a, \Dl_0,\Dl_{s+1}, W_s\big)\big) 
\cdot A\big(\sum_{a=1}^{s+1}e_a, \Dl_0,\Dl_{s+1}, W_s\big)^{-1} \cdot
\\
&
\cdot\, 
D_u\big(A\big(\sum_{a=1}^{s+1}e_a, \Dl_0,\Dl_{s+1}, W_s\big)\big)
\cdot A\big(\sum_{a=1}^{s+1}e_a, \Dl_0,\Dl_{s+1}, W_s\big)^{-1} 
\equiv
\\
&
\notag
\equiv
D_u\big(D_v\big(A\big(\sum_{a=1}^{s}e_a, \Dl_0,\Dl_{s}, W_{s-1}\big)\big)\big)
\cdot A\big(\sum_{a=1}^{s}e_a, \Dl_0,\Dl_{s}, W_{s-1}\big)^{-1} +
\\
&
+\,
D_v\big(A\big(\sum_{a=1}^{s}e_a, \Dl_0,\Dl_{s}, W_{s-1}\big)\big)
\cdot A\big(\sum_{a=1}^{s}e_a, \Dl_0,\Dl_{s}, W_{s-1}\big)^{-1}\cdot
\\&
\cdot\,
D_u\big(A\big(\sum_{a=1}^{s}e_a, \Dl_0,\Dl_{s}, W_{s-1}\big)\big)
\cdot A\big(\sum_{a=1}^{s}e_a, \Dl_0,\Dl_{s}, W_{s-1}\big)^{-1}
\eea
modulo $p^s$.
It remains to apply \eqref{Der} with $D=D_u$ and $D=D_v$ and  
$\ell=0$ to see that the second terms on both sides agree modulo~$p^s$.
After their cancellation we are left with the required congruences in~\eqref{Der2}.
\end{proof}

\begin{rem}
The results of Section \ref{sec 2} in the case  
$ {\bf e}=(e_1, \dots, e_{l})=(1,\dots,1)$ and 
$\Dl=(\Dl_0, \dots,  \Dl_l)$ such that $\Dl_0 =\dots =  \Dl_l$
were obtained in \cite{VZ2}.

\end{rem}

\section{Convergence}
\label{sec 3}

\subsection{Unramified extensions of $\Q_p$}

We fix  an algebraic closure $\overline{\Q_p}$ of $\Q_p$.
For every $m$, there is a unique unramified extension of $\Q_p$ in 
 $\overline{\Q_p}$ of degree $m$, denoted by $\Q_p^{(m)}$.
This can be obtained by attaching to $\Q_p$ a primitive root of $1$ of order $p^m-1$.
The norm $|\cdot|_p$ on $\Q_p$ extends to a norm $|\cdot|_p$ on 
$\Q_p^{(m)}$.
Let 
\bea
\Z_p^{(m)} = \{ a\in \Q_p^{(m)} \mid |a|_p\leq 1\}
\eea
denote the ring of integers in $\Q_p^{(m)}$. The ring $\Z_p^{(m)}$
has the unique maximal ideal 
\bea
\mathbb M_p^{(m)} = \{ a\in \Q_p^{(m)} \mid |a|_p <1\},
\eea
such that $\mathbb Z_p^{(m)}\big/ \mathbb M_p^{(m)}$ is isomorphic to the finite field
$\F_{p^m}$.

For every $u\in\F_{p^m}$ there is a unique $\tilde u\in \mathbb Z_p^{(m)}$ that is a lift of $u$ and such that 
$\tilde u^{p^m}=\tilde u$. The element $\tilde u$ is called the Teichmuller lift of $u$.

\subsection{Domain $\frak D_B$}
For $u\in\F_{p^m}$ and $r>0$ denote
\bea
D_{u,r} = \{ a\in \Z_p^{(m)}\mid |a-\tilde u|_p<r\}\,.
\eea
We have the partition
\bea
\Z_p^{(m)} = \bigcup_{u\in\F_{p^m}} D_{u,1}\,.
\eea

Recall  $z=(z_1,\dots,z_n)$. For $B(z) \in \Z[z]$, define
\bea
\frak D_B \ =\ \{ a\in (\Z_p^{(m)})^n\,  \mid \  |B(a)|_p=1\} .
\eea
Let $\bar B(z)$ be the projection of $B(z)$ to $\F_p[z]\subset \F_{p^m}[z]$.
Then $\frak D_B$ is the union of unit polydiscs,
\bea
\frak D_B = \bigcup_{\substack{u_1,\dots,u_n\in \F_{p^m}\\  \bar B(u_1,\dots, u_n)\ne 0}} \  D_{u_1,1}\times \dots
\times D_{u_n,1}\,.
\eea
For any $k$ we have
\begin{align}
\notag
\{ a\in (\Z_p^{(m)})^n \mid \ |B(a^{p^k})|_p=1\}
&=\bigcup_{\substack{u_1,\dots,u_n\in \F_{p^m}\\  \si^k(\bar B(u_1,\dots, u_n))\ne 0}} \  D_{u_1,1}\times \dots
\times D_{u_n,1} =
\\
\notag
&=\bigcup_{\substack{u_1,\dots,u_n\in F_{p^m}\\  \bar B(u_1,\dots, u_n)\ne 0}} \  D_{u_1,1}\times \dots
\times D_{u_n,1} =
 \frak D_B \,.
\end{align}

\begin{lem}
[{\cite[Lemma 6.1]{VZ2}}]
\label{lem nonempty}

Let $\bar B_1(z), \dots, \bar B_k(z) \in  \F_p[z]$
be nonzero polynomials  such that 
$\deg \bar B_j(z)\leq d$, $j=1,\dots,k$,  for some $d$.
If $kd+1< p^m$, then  the set
\bea
 \{ a\in (\F_{p^m})^n\mid \bar B_j(a) \ne 0, \, j=1,\dots, f\}
\eea 
is nonempty.

\end{lem}

\subsection{Uniqueness theorem}

Let  $\frak D \subset (\Z_p^{(m)})^n$ be the union of some of the
unit polydiscs
\linebreak
$ D_{u_1,1}\times \dots \times D_{u_n,1}$\,, where $u_1,\dots,u_n\in\F_{p^m}$.

Let $(F_i(z))_{i=1}^\infty$ and $(G_i(z))_{i=1}^\infty$ be two sequences of rational functions
on $(\F_{p^m})^n$. Assume that each of the rational functions has the form 
$P(z)/Q(z)$, where $P(z), Q(z)\in\Z[z]$, and for any
polydisc $ D_{u_1,1}\times \dots \times D_{u_n,1}\,\subset \frak D$, we have
$|Q(\tilde u_1,\dots,\tilde u_n)|_p=1$,  which implies that
\bea
|Q(a_1,\dots,a_n)|_p=1,\qquad \forall\ (a_1,\dots,a_n)\in \frak D.
\eea
Assume that  the sequences 
$(F_i(z))_{i=1}^\infty$ and $(G_i(z))_{i=1}^\infty$
 uniformly converge on $\frak D$ to analytic functions, which we denote by $F(z)$ and $G(z)$, respectively. 

\begin{thm}
[\cite{VZ2}]

\label{thm U}
Under these assumptions, if $F(z)=G(z)$ on an open 
nonempty subset of~$\frak D$. Then $F(z)=G(z)$ on~$\frak D$.

\end{thm}

\subsection{Infinite tuples}

Let ${\bf e}=(e_1, e_2, \dots)$ be an infinite tuple  of positive integers.
Let $\Dl=(\Dl_0, \dots,  \Dl_l)$ be an infinite tuple  of nonempty finite subsets of $\Z^r$  of the same size
$\# \Dl_j=g$ for some  positive integer $g$. Let  $\La=(\La_0(x),\La_1(x), \dots)$ 
be an infinite tuple of Laurent polynomials in $\Z_p[x^{\pm1}]=\Z_p[t^{\pm1},z^{\pm1}]$.

\vsk.2>

Assume that the tuple $\La$ is $(\Dl,{\bf e})$-admissible.

\vsk.2>

Assume that each of the tuples ${\bf e}, \Dl, \La$ have only finitely many distinct elements. This means that 
there is a finite set of 4-tuples 
\bean
\label{tuples}
\mc T = \{(e^j, \bar \Dl^j, \tilde D^j,\La^j)  \mid j=1,\dots,k\}
\eean
such that for any $l\geq 0$ the 4-tuple $(e_{l+1}, \Dl_l,\Dl_{l+1}, \La_l)$ equals one of the 4-tuples
in $\mc T$.

\begin{defn}
\label{def F}
The $(\Dl,{\bf e})$-admissible tuple $\La$ is called \emph{nondegenerate}, if
 for any $i=1,\dots,k$, the Laurent polynomial 
 \bea
 \det A\big (e^j, \bar\Dl^j,\tilde \Dl^j, \La^j\big) \ \in \  \Z_p[z^{\pm1}]
 \eea
  is nonzero modulo~$p$. 

\end{defn}

Recall the notation:
\bea
W_s(x)
&:=&
\La_0(x)\La_1(x)^{p^{e_1}}\cdots \La_s(x)^{p^{e_1+\dots+e_s}},
\\
W_s^{(j)}(x)
&:=&
\La_j(x)\La_{j+1}(x)^{p^{e_{j+1}}}\cdots 
\La_s(x)^{p^{e_{j+1}+\dots +e_s}}.
\eea
If a $(\Dl,{\bf e})$-admissible tuple $\La$ is nondegenerate,  then
for any $0\leq j\leq s$,  the Laurent polynomials 
$\det A\big(\sum_{a=j+1}^{s+1}e_a, \Dl_j,\Dl_{s+1}, W_s^{(j)}\big)\in \Z_p[z^{\pm1}]$ 
are not congruent to zero modulo~$p$ and we may consider congruences involving the inverse matrices
$A\big(\sum_{a=j+1}^{s+1}e_a, \Dl_j,\Dl_{s+1}, W_s^{(j)}\big)^{-1}$.

\subsection{Domain of convergence}

Assume that $\La$ is an infinite nondegenerate  $(\Dl,{\bf e})$-admissible tuple and
 $m$ is a positive integer.
Denote

\bea
\frak D^{(m)}
= \{a \in (\Z_p^{(m)})^{n} \ \mid \ 
|\det A\big (e^j, \bar\Dl^j,\tilde \Dl^j, \La^j(t,a) \big)|_p=1,
 \,\,j=1,\dots,k\}.
\eea

\vsk.2>

\begin{lem}
\label{lem |det|}
For any $0\leq j\leq s$ and $a\in \frak D^{(m)}$ we have
\bea
\Big\vert\det 
A\Big({\sum}_{a=j+1}^{s+1}e_{a},\Dl_j,\Dl_{s+1}, W_s^{(j)}(t,a)\Big)
\Big\vert_p =1.
\eea
\qed
\end{lem}

\begin{cor}
All entries of $A\big(\sum_{a=j+1}^{s+1}e_{a},\Dl_j,\Dl_{s+1}, W_s^{(j)}(t,z)\big)^{-1}$
 are rational functions in $z$ regular on $\frak D^{(m)}$.  
For every $a\in\frak D^{(m)}$  all entries of 
$A\big(\sum_{a=j+1}^{s+1}e_{a},\Dl_j,\Dl_{s+1}, W_s^{(j)}(t,a)\big)$
 and 
 $A\big(\sum_{a=j+1}^{s+1}e_{a},\Dl_j,\Dl_{s+1}, W_s^{(j)}(t,a)\big)^{-1}$ are elements of $\Z_p^{(m)}$.
\qed

\end{cor}

\vsk.2>

\begin{thm}
\label{thm conv}
Let $\La$ be an infinite nondegenerate
$(\Dl,{\bf e})$-admissible tuple.  Consider the sequence of $g\times g$ matrices 
\bean
\label{sec of m}
\phantom{aaa}
\Big( A\Big({\sum}_{a=1}^{s+1}e_{a},\Dl_0,\Dl_{s+1}, W_s(t,z)\Big)
\cdot
\si^{e_1}\Big(A\Big({\sum}_{a=2}^{s+1}e_{a},\Dl_1,\Dl_{s+1}, W_s^{(1)}(t,z)\Big)\Big)^{-1} \Big)_{s\geq 0}
\eean
whose entries are rational functions in $z$ regular on the domain $\frak D^{(m)}$.
This sequence uniformly converges on $\frak D^{(m)}$ as $s\to\infty$ to an analytic
$g\times g$ matrix with values in $\Z_p^{(m)}$. 
Denote this matrix by  $\mc A_\La(z)$.
For $a\in\frak D^{(m)}$ we have
\bean
\label{det 1}
\Big\vert \det \mc A_\La(a) \Big\vert_p=1\,
\eean
and the matrix $\mc A_\La(a)$ is invertible. 
\end{thm}

\vsk.2>

\begin{proof}
By part (i) of Theorem \ref{thm 1.6} we have
$\vert\det \si^{e_1}\big(A\big({\sum}_{a=2}^{s+1}e_{a},\Dl_1,\Dl_{s+1}, W_s^{(1)}(t,a)\big)\big)\vert_p =1$
 for $a\in\frak D^{(m)}$.
Hence the matrix in \eqref{sec of m} is a matrix of rational functions
in $z$ regular on $\frak D^{(m)}$. Moreover, if $a\in\frak D^{(m)}$,
 then every entry of this matrix is an element of $\Z_p^{(m)}$.
The uniform convergence on $\frak D^{(m)}$ of the sequence \eqref{sec of m}
is a corollary of part (ii) of Theorem \ref{thm 1.6}.
Equation \eqref{det 1} follows from part (i) of Theorem \ref{thm 1.6}.
The theorem is proved.
\end{proof}

\vsk.2>

\begin{thm}
\label{thm conv2}
Let $\La$ be an infinite nondegenerate
$(\Dl,{\bf e})$-admissible tuple, and
$D=D_v$, $v=1,\dots,n$.
  Given $\ell\geq 0$ consider the sequence of $g\times g$ matrices 
\bea
\Big(\,D\Big(\si^\ell\Big(A\Big({\sum}_{a=1}^{s+1}e_a, \Dl_0,\Dl_{s+1}, W_s\Big)\Big)\Big)
\cdot \si^\ell\Big(A\Big({\sum}_{a=1}^{s+1}e_a, \Dl_0,\Dl_{s+1}, W_s\Big)\Big)^{-1}\,
\Big)_{s\geq 0}
\eea
whose entries are rational functions in $z$ regular on the domain $\frak D$.
This sequence uniformly converges on $\frak D$ as $s\to\infty$ to an analytic
$g\times g$ matrix with values in $\Z_p^{(m)}$. 
Denote this matrix by $\mc A_{\La,D\si^\ell}(z)$.

\end{thm}

\begin{proof} 
The theorem is a corollary of Theorem \ref{thm der}.
\end{proof}

\begin{thm}
\label{thm conv3}
Let $\La=(\La_0(x), \La_1(x), \La_2(x),\dots )$ be an infinite nondegenerate
$(\Dl,{\bf e})$-admis\-sible tuple.
  Given $\ell\geq 0$ and $1\leq u,v\leq n$, consider the sequence of $g\times g$ matrices 
\bea
\Big(\,D_u\Big(D_v\Big(A\Big({\sum}_{a=1}^{s+1}e_a, \Dl_0,\Dl_{s+1}, W_s\Big)\Big)\Big)
\cdot A\Big({\sum}_{a=1}^{s+1}e_a, \Dl_0,\Dl_{s+1}, W_s\Big)^{-1}\,\Big)_{s\geq 0}
\eea
whose entries are rational functions in $z$ regular on the domain $\frak D$.
This sequence uniformly converges on $\frak D$ as $s\to\infty$ to an analytic
$g\times g$ matrix with values in $\Z_p^{(m)}$. 
Denote this matrix by $\mc A_{\La,D_uD_v}(z)$.

\end{thm}

\begin{proof} 
The theorem is a corollary of Theorem \ref{thm der2}.
\end{proof}

\vsk.2>

Let $\La=(\La_0(x), \La_1(x), \La_2(x),\dots )$ be an infinite nondegenerate
$(\Dl,{\bf e})$-admissible tuple.
 Consider the $g\times g$ matrix valued functions
$\mc A_{\La,\frac{\der}{\der z_u}\si^0}(z)$, $\mc A_{\La,\frac{\der}{\der z_v}\si^0}(z)$
in Theorem \ref{thm conv2} and denote them by $\mc A_u(z)$, $\mc A_v(z)$, respectively.
Consider the $g\times g$ matrix valued function
$\mc A_{\La,\frac{\der}{\der z_u}\frac{\der}{\der z_v}}(z)$ in Theorem \ref{thm conv3}
and denote it by $\mc A_{u,v}(z)$.
All the three functions are analytic on $\frak D^{(m)}$.

\begin{lem}
[{\cite[Lemma 3.7]{VZ2}}]
\label{lem conv3}
We have
\bea
\frac{\der}{\der z_u}\mc A_v = \mc A_{u,v} - \mc A_v\mc A_u\,.
\eea
\qed
\end{lem}

\section{\KZ/ equations and complex solutions}
\label{sec 4}

\subsection{\KZ/ equations}
Let $\g$ be a simple Lie algebra with an invariant scalar product.
The { Casimir element}  is  $\Om = {\sum}_i \,h_i\ox h_i  \in  \g \ox \g$,
where $(h_i)\subset\g$ is an orthonormal basis.
Let  $V=\otimes_{i=1}^n V_i$ be 
a tensor product of $\g$-modules, $\ka\in\C^\times$ a nonzero number.
The {\it  \KZ/ equations} is the system of differential 
equations on a $V$-valued function $I(z_1,\dots,z_n)$,
\bea
\frac{\der I}{\der z_i}\ =\ \frac 1\ka\,{\sum}_{j\ne i}\, \frac{\Om_{i,j}}{z_i-z_j} I, \qquad i=1,\dots,n,
\eea
where $\Om_{i,j}:V\to V$ is the Casimir operator acting in the $i$th and $j$th tensor factors,
see \cite{KZ, EFK}.

\vsk.2>

This system is a system of Fuchsian first order
 linear differential equations. 
  The equations are defined on the complement in $\C^n$ to the union of all diagonal hyperplanes.
 
\vsk.2>

The object of our discussion is the following particular case.  Let $n, q$ be positive integers.
 We consider the following system of differential and algebraic  equations
for a column $n$-vector $I=(I_1,\dots,I_n)$ depending on variables $z=(z_1,\dots,z_n)$\,:
\bean
\label{KZ}
\phantom{aaa}
 \frac{\partial I}{\partial z_i} \ = \
   {\frac 1 q} \sum_{j \ne i}
   \frac{\Omega_{ij}}{z_i - z_j}  I ,
\quad i = 1, \dots , n,
\qquad
I_1+\dots+I_{n}=0,
\eean
where $z=(z_1,\dots,z_n)$;
the $n\times n$-matrices $\Om_{ij}$ have the form
\bea
 \Omega_{ij} \ = \ \begin{pmatrix}
             & \vdots^{\kern-1.2mm i} &  & \vdots^{\kern-1.2mm j} &  \\
        {\scriptstyle i} \cdots & {-1} & \cdots &
            1   & \cdots \\
                   & \vdots &  & \vdots &   \\
        {\scriptstyle j} \cdots & 1 & \cdots & -1&
                 \cdots \\
                   & \vdots &  & \vdots &
                   \end{pmatrix} ,
\eea
and all other entries are zero.
 This  joint system of {\it differential and 
algebraic equations} will be called the {\it system of \KZ/  equations} in this paper.

     \vsk.2>
     
For $i=1,\dots,n$ denote
\bean
\label{GH}
&
H_i(z) =   {\frac 1q} \sum_{j\ne i}    \frac{\Omega_{ij}}{z_i - z_j}\,,
\qquad
\nabla_i^{\on{KZ}} = \frac{\der}{\der z_i} - H_i(z), \qquad i=1,\dots,n.
\eean
The linear operators $H_i(z)$ are called the Gaudin Hamiltonians.
The \KZ/ equations can be written as the system of equations,
\bea
\nabla_i^{\on{KZ}}I=0, \quad i=1,\dots,n,\qquad I_1+\dots + I_n =0.
\eea
\vsk.2>

System  \eqref{KZ} is the system of the  differential \KZ/ equations with 
parameter $\ka=q$ associated with the Lie algebra $\sll_2$ and the subspace of singular vectors of weight 
$n-2$ of the tensor power 
$(\C^2)^{\ox {n}}$ of two-dimensional irreducible $\sll_2$-modules, up to a gauge transformation, see 
this example in  \cite[Section 1.1]{V2}, see also \cite{V3}.

\subsection{Solutions over $\C$}
\label{sec 11.4}

Define the {\it master function}
\bea
\Phi(t,z) = (t-z_1)^{-1/q}\dots (t-z_n)^{-1/q}
\eea
and the column $n$-vector
\bean
\label{KZ sol} 
I^{(C)}(z) = (I_1,\dots,I_n):=
\int_{C}
\Big(\frac {\Phi(t,z)}{t-z_1}, \dots , \frac {\Phi(t,z)}{t-z_n}\Big)dt
\,,
\eean
where  $C\subset \C-\{z_1,\dots,z_n\}$  
is a contour on which the integrand  takes its initial value when $t$ encircles $C$.

\begin{thm}
The function $I^{(C)}(z)$ is a solution of system \eqref{KZ}.

\end{thm}

This theorem is a very particular case of the results in \cite{SV1}.

\begin{proof}  
The theorem follows from  Stokes' theorem and the two identities:
\bean
\label{i1}
-\frac 1q\,
\Big(\frac {\Phi(t,z)}{t-z_1}
 + \dots  + \frac {\Phi(t,z)}{t-z_n}\Big)\,  
=\, \frac{\der\Phi}{\der t}(t,z)\,,
\eean
\bean
\label{i2}
\Big(\frac{\der }{\der z_i}-\frac1q
\sum_{j\ne i} \frac {\Omega_{i,j}}{z_i-z_j} \Big)
\Big(\frac {\Phi(t,z)}{t-z_1}, \dots, \frac {\Phi(t,z)}{t-z_n}\Big)\,  
= \frac{\der \Psi^i}{\der t} (t,z),
\eean
where  $\Psi^i(t,z)$ is the column $n$-vector   $(0,\dots,0,-\frac{\Phi(t,z)}{t-z_i},0,\dots,0)$ with 
the nonzero element at the $i$-th place. 
\end{proof}

\begin{thm} [{cf.~\cite[Formula (1.3)]{V1}}]
\label{thm dim}

All solutions of system \eqref{KZ} have this form. 
Namely, the complex vector space of solutions of the form \eqref{KZ sol} is $(n-1)$-dimensional.

\end{thm}

\subsection{Solutions as vectors of first derivatives}
\label{sec 11.5}

Consider the integral
\bea
T(z) = T^{(C)}(z) =
\int_C
\Phi(t,z) \,dt.
\eea
Then
\bea
I^{(C)}(z) 
=
\,
q\,
\Big(\frac {\der T^{(C)}}{\der z_1}, \dots ,
\frac {\der T^{(C)}}{\der z_n}\Big).
\eea
Denote
$\nabla T =
\Big(\frac {\der T}{\der z_1},\dots,  \frac {\der T}{\der z_n}\Big)$.
Then the column gradient vector $\nabla T$ of the function $T(z)$ satisfies the following system of  \KZ/ equations
\bea
\nabla_i^{\on{KZ}} \nabla T =0, \quad i=1,\dots,n,\qquad  
\frac {\der T}{\der z_1} +\dots + \frac {\der T}{\der z_n}=0.
\eea
This is a system of second order linear differential equations on the function $T(z)$.

\section{Solutions modulo powers of $p$}
\label{sec 5}

\subsection{Assumptions}
\label{sec ass}
${}$

Let $p, q$, $p>q$, be prime numbers.  Let $e$ be the order of $p$ modulo $q$, 
that is, the least positive integer such that
$p^e\equiv 1\pmod{q}$. Hence $(p^e-1)/q$ is a positive integer.
Let $n=gq+1$ for some positive integer $g$. 
Assume that $p^e> n$ and $p\geq n+q-2$.

In this paper we consider the system of \KZ/ equations
\eqref{KZ} with $n=gq+1$ and $\ka = q$ and study polynomial solutions of the \KZ/ equations
modulo powers of $p$.

\subsection{Polynomial solutions}
\label{sec:new}

For an integer $s\geq 1$ define the {\it master polynomial}
\bea
\Phi_s(t,z) = \big((t-z_1)\dots(t-z_n)\big)^{(p^{es}-1)/q}.
\eea
 For $\ell=1,\dots, g$ define the column $n$-vector
\bea
I_{s,\ell} (z)=(I_{s,\ell,1}, \dots, I_{s,\ell.n})
\eea
as the coefficient of $t^{\ell p^{es}-1}$ in the column $n$-vector of polynomials 
$\big(\frac{\Phi_s(t,z)}{t-z_1}, \dots, \frac {\Phi_s(t,z)}{t-z_n}\big)$.
Notice that 
\bea
\deg_t \frac{\Phi_s(t,z)}{t-z_i} = (gq+1)\frac{p^{es}-1}q-1
= gp^{es}-1 + \frac{p^{es}-1}q -g.
\eea
If $\ell >g$,  then the polynomial
$\frac{\Phi_s(t,z)}{t-z_i}$ does not have the monomial 
$t^{\ell p^{es}-1}$.

\vsk.2>

\begin{thm} [cf. \cite{V5, VZ2}]
\label{thm 7.3}
The column $n$-vector  $I_{s,\ell}(z)$ of polynomials in $z$ is a solution of the system of \KZ/ equations
\eqref{KZ}  modulo $p^{es}$.

\end{thm}

\vsk.2>

We call the column $n$-vectors  $I_{s,\ell}(z)$, $\ell=1,\dots,g$,
 the {\it $p^{es}$-hypergeometric solutions} of the \KZ/ equations
\eqref{KZ}.

\vsk.2>
\begin{proof}

We have the following modifications of identities \eqref{i1}, \eqref{i2}\,:
\bea
\frac {p^{es}-1}q\,
\Big(\frac {\Phi_s(t,z)}{t-z_1}
 + \dots + \frac {\Phi_s(t,z)}{t-z_n}\Big)\,  
=\, \frac{\der\Phi_s}{\der t}(t,z)\,,
\eea
\bea
\Big(\frac{\der }{\der z_i} + \frac {p^{es}-1}q
\sum_{j\ne i} \frac {\Omega_{i,j}}{z_i-z_j} \Big)
\Big(\frac {\Phi_s(t,z)}{t-z_1}, \dots, \frac {\Phi_s(t,z)}{t-z_n}\Big)\,  
= \frac{\der \Psi_s^i}{\der t} (t,z),
\eea
where  $\Psi_s^i(t,z)$ is the column $n$-vector   $(0,\dots,0,-\frac{\Phi_s(t,z)}{t-z_i},0,\dots,0)$ with 
the nonzero element at the $i$-th place. Theorem \ref{thm 7.3} follows from these identities.
\end{proof}

Consider the $n\times g$ matrix
\bea
I_s(z) = (I_{s,1},\dots,I_{s,g}) = \big(I_{s, \ell, i}\big)_{\ell = 1,\dots,g}^{i=1,\dots,n}\ ,
\eea
where $I_{s, \ell, i}$ stays at the $\ell$-th column and $i$-th row. The matrix $I_s(z)$ satisfies the \KZ/ equations,
\bea
\nabla_i^{\on{KZ}}I_s(z) =0, \quad i=1,\dots,n,\qquad I_{s,\ell,1}+\dots + I_{s,\ell,n}(z) =0, 
\quad \ell=1,\dots,g,
\eea
modulo $p^{es}$.

\subsection{Coefficients of solutions}

Consider the lexicographical ordering of monomials
\linebreak
$z_1^{d_1}\dots z_{n}^{d_{n}}$. We  have $z_1>\dots > z_{n}$ and  so on.
For a nonzero Laurent polynomial
$f(z)=\sum_{d_1,\dots,d_{n}} a_{d_1,\dots,d_{n}} z_1^{d_1}$
\dots $z_{n}^{d_{n}}$\, with coefficients in $\Z$\,,\
the nonzero summand 
$a_{d_1,\dots,d_{n}} z_1^{d_1}\dots z_n^{d_{n}}$ with the largest monomial
$z_1^{d_1}$ \dots $z_{n}^{d_{n}}$ is called
 the { leading term} of $f(z)$.
 
 \vsk.2>
If $f(z)$ and $g(z)$ are two nonzero Laurent polynomials, then
the leading term of $f(z)g(z)$ equals the product of the leading terms of $f(z)$ and $g(z)$.

\begin{lem}
\label{lem leco}
For  $l=1,\dots,g$,  the leading term of the vector-polynomial $I_{1,\ell}$ equals
\linebreak
$C_{\ell}\cdot (z_1\dots z_{q(g-\ell)+1})^{(p^e-1)/q}/z_{q(g-\ell)+1}^\ell$,
\bean
\label{leco}
C_{\ell} \,=\,
\pm\binom{(p^e-1)/q-1}{\ell-1}\Big(0, \dots, 0, 1,\frac{p^e-1}{q\ell},\dots, \frac{p^e-1}{q\ell}\Big),
\eean
where $\frac{p^e-1}{q\ell}$ is repeated $q\ell$ times, and
\bean
\label{bino}
\binom{(p^e-1)/q-1}{\ell-1}\not\equiv 0 \pmod{p}.
\eean

\end{lem}

\begin{proof}
Formula \eqref{leco} is obtained by inspection. To prove \eqref{bino}
consider the $p$-ary presentation  $(p^e-1)/q-1 = a_0 + a_1p +\dots$ with
$0\leq a_i\leq p-1$.  The inequality  \eqref{bino} follows from the inequality
$a_0\geq g-1$ and Lucas theorem.

We prove that $a_0\geq g-1$ under our assumption $p\geq n+q-2$.  Indeed, 
$p^e = 1 + q(1+a_0) + qa_1p+ \dots$. Hence
$1+ q(1+a_0)\geq p$. Let $p=qk+r$ for some integers $k,r$, $1\leq r\leq q-1$. Then
$1+q(1+a_0)\geq qk+r$ or
$q(1+a_0)\geq qk+r-1\geq kq$ or $a_0\geq k-1$. Hence 
$a_0\geq g-1$ if  $k\geq g$.

The inequality 
$p\geq n+q-2$ can be written as
$kq+r\geq gq+1+q-2$ or
$kq\geq gq+ q-r-1$. Hence $k\geq g$. The lemma is proved.
\end{proof}

\begin{lem}
\label{lem minor}

Consider the $n\times g$ matrix
$I_1(z) = (I_{1,1},\dots,I_{1,g})$ and its $g\times g$ minor $M(z)$ in rows with
indices $q(g-\ell) +1$ where $\ell=1,\dots,g$.
Then  $M(z)$ is a homogeneous polynomial of degree
\bean
\label{deg M}
d_M  = \frac{p^e-1}{q}\cdot\frac{qg^2+2g-qg}2 - \frac{g(g+1)}2\,,
\eean
and the polynomial $M(z)$ is nonzero modulo $p$.

\end{lem}

\begin{proof}
Every column of $I_{1,\ell}$ is a homogeneous polynomial. Hence $M(z)$ is a homogeneous polynomial.
By Lemma \ref{lem leco} the leading term of $M(z)$ equals
\bean
\label{lcd}
\pm \prod_{\ell=1}^g \binom{(p^e-1)/q-1}{\ell-1}
(z_1\dots z_{q(g-\ell)+1})^{(p^e-1)/q}/z_{q(g-\ell)+1}^\ell\,.
\eean
This expression is  nonzero modulo $p$ by Lemma \ref{lem leco}. Formula \eqref{lcd} implies \eqref{deg M}.
\end{proof}

\section{Congruences for solutions of \KZ/ equations}
\label{sec 6}
\subsection{Congruences for Hasse--Witt matrices of \KZ/ equations}
Let $r=1$, $n=gq+1$, ${\bf e}=(e,e,\dots)$, where $e$ is defined in Section \ref{sec ass}. Let
\bean
\label{DN}
&
\Ga =\{1,\dots,g\}\subset \Z, \qquad
\Dl=(\Ga,\Ga, \dots ),
\\
\notag
&
N=[0, gp^e + (p^e-1)/q-g]\subset \R.
\eean
The infinite tuple $(N, N,\dots)$ of intervals is $(\Dl,{\bf e})$-admissible\textup, see Definition \textup{\ref{defN}}.

Recall the polynomial
\bea
 \Phi_1(t,z) = \big((t-z_1)\dots(t-z_n)\big)^{(p^e-1)/q}.
\eea
 The Newton polytope of $\Phi_1(t,z)$ with respect to variable $t$ is the interval
 \linebreak
$N=[0, gp + (p-1)/q -g]$. We also have 
\bea
\Phi_s(t,z) = \Phi_1(t,z)\cdot \Phi_1(t,z)^{p^e}\dots \Phi_1(t,z)^{p^{e(s-1)}}\,.
\eea
The infinite tuple $(\Phi_1(t,z), \Phi_1(t,z),\dots)$ is $(\Dl,{\bf e})$-admissible, see Definition \ref{defn}.

\vsk.2>

For $s\geq 1$ consider the Hasse--Witt $g\times g$ matrix 
\bea
A(\Phi_s(t,z)) := A(es,  \Ga, \Ga, \Phi_s(t,z))
=
\big( \Cf_{p^{es}v-u}(\Phi_s(t,z))\big)_{u,v=1,\dots,g}\,,
\eea
see \eqref{Cuv+}. The entries of this matrix are polynomials in $z$.

\vsk.2>

\begin{thm} 
\label{thm F ne}

The determinant $\det A(\Phi_1(t,z))$ is a homogeneous polynomial in $z$ of degree
\bean
\label{deg d}
d_\Phi = \frac{p^e-1}{q}\cdot\frac{qg^2+2g-qg}2\,,
\eean
and the determinant is nonzero modulo $p$.

\end{thm}

\begin{proof}

Denote $A(\Phi_1(t,z)) = : (A_{u,v}(z))_{u,v=1,\dots,g}$\,.

\begin{lem}
The leading term of $A_{u,v}(z)$ equals 
\bea
&&
\pm\binom{(p^e-1)/q}{v-u} 
(z_1z_2\dots z_{qg+1-qv})^{(p^e-1)/q}/ z_{qg+1-qv}^{v-u}\,, \qquad \text{if}\;\; v\geq u,
\\
\notag
&&
\pm\binom{(p^e-1)/q}{u-v} 
(z_1z_2\dots z_{qg+1-qv})^{(p^e-1)/q} z_{qg+2-qv}^{u-v}\,, \qquad \quad\ \text{if}\;\; v\leq u.
\eea
For example, for $g=2$ the matrix of leading terms is
\bean
\label{ex d}
\begin{pmatrix}
 \pm(z_1\dots z_{g+1})^{(p^e-1)/q} & \pm \binom{(p^e-1)/q}{1}z_1^{(p^e-1)/q}/z_1
 \\
   \pm \binom{(p^e-1)/q}{1}(z_1\dots z_{q+1})^{(p^e-1)/q}z_{q+2} & \pm z_1^{(p^e-1)/q} 
\end{pmatrix} .
\eean
\end{lem}

\begin{proof}
The proof is by inspection.
\end{proof}

The fact that $\det A(\Phi_1(t,z))$ is a homogeneous polynomial easily follows from
the definition of $A(\Phi_1(t,z))$. It is also easy to see that the leading term of the 
determinant of the matrix of leading terms
 of $A_{u,v}(z)$ equals the product of diagonal elements,
\bean
\label{Deg}
\pm \,\prod_{v=1}^{g}( z_1\dots z_{qg+1-qv})^{(p^e-1)/q}.
\eean
This expression is not congruent to zero modulo $p$.
Counting the degree of the monomial in \eqref{Deg} we 
obtain \eqref{deg d}.
This proves Theorem \ref{thm F ne}.
\end{proof}

\begin{cor}
\label{cor F ne}
The infinite nondegenerate $(\Dl, {\bf e})$-admissible 
tuple $(\Phi_1(t,z), \Phi_1(t,z),\dots )$ satisfies the assumptions of Theorem \textup{\ref{thm 1.6}}.
Therefore,
\begin{enumerate}
\item[\textup{(i)}] for $s\geq 1$ we have
\begin{align}
\label{alal+}
A(\Phi_{s}(t,z))
\equiv
A(\Phi_1(t,z))\cdot\si^e(A(\Phi_1(t,z)))\cdots \si^{e(s-1)}(A\big(\Phi_1(t,z)))
\pmod{p}\,;
\end{align}
\item[\textup{(ii)}] for $s\geq 1$ 
the determinant of the matrix $A(\Phi_s(t,z))$
is a polynomial, which is nonzero modulo~$p$, and 
we have modulo $p^s$\,\textup:
\begin{align*}
&
A(\Phi_{s+1}(t,z))\cdot \si^e(A(\Phi_{s}(t,z)))^{-1}
\equiv
A(\Phi_{s}(t,z))\cdot \si^e(A(\Phi_{s-1}(t,z)))^{-1},
\end{align*}
where for $s=1$ we understand the second factor on the right-hand side 
as the  $g\times g$ identity matrix.
\end{enumerate}

\end{cor}

\begin{proof}
The corollary follows from Theorems \ref{thm F ne} and \ref{thm 1.6}.
\end{proof}

\vsk.2>

\vsk.2>

\subsection{Congruences for frames of solutions of \KZ/ equations}

\begin{thm}
\label{thm coS}
We have the following congruences of $n\times g$ matrices.

\begin{enumerate}
\item[\textup{(i)}] For $s\geq 1$,
\bea
I_{s+1}(z) \cdot A( \Phi_{s+1}(t,z))^{-1}
\equiv 
I_{s}(z) \cdot A(\Phi_{s}(t,z))^{-1}
\pmod{p^s}\,.
\eea

\item[\textup{(ii)}] For $s\geq 1$ and $j=1,\dots, n$,
\bea
\frac {\der I_{s+1}}{\der z_j} (z) \cdot A(\Phi_{s+1}(t,z))^{-1}
\equiv 
\frac{\der I_{s}}{\der z_j}(z) \cdot A(\Phi_{s}(t,z))^{-1}
\pmod{p^s}\,.
\eea

\end{enumerate}

\end{thm}

\vsk.2>

\begin{proof}
Consider the first row of the Hasse--Witt matrix $A(\Phi_s(t,z))$,
\bea
\big(A_{1,1}(\Phi_s(t,z)),\dots, A_{1,g}(\Phi_s(t,z))\big),
\quad
A_{1,\ell}(\Phi_s(t,z)) = \on{Coeff}_{\ell p^s-1}(\Phi_{s}(t,z)).
\eea
For each  $A_{1,\ell}(\Phi_s(t,z))$ we view the gradient
\bea
\nabla A_{1,\ell}(\Phi_s(t,z))=\Big(\frac{\der A_{1,\ell}(s)}{\der z_1}, \dots,
\frac{\der A_{1,\ell}(s)}{\der z_n} \Big)
\eea
as a column $n$-vector. The resulting $n\times g$ matrix of gradients
\bea
\nabla A(s,z):=(\nabla A_{1,1}(\Phi_s(t,z)),\dots,\nabla A_{1,g}(\Phi_s(t,z)))
\eea
is proportion to the matrix $I_s(z)$, $\nabla A(s,z) = \frac{1-p^{es}}q I_s(z)$.
By Theorems \ref{thm der} and \ref{thm der2} we have
modulo $p^s$,
\bea
&
\nabla A(s+1,z)
\cdot A(\Phi_{s+1}(t,z))^{-1}
\equiv 
\nabla A(s,z) \cdot A(\Phi_{s}(t,z))^{-1},
\\
&
\frac{\der}{\der z_j}\big(\nabla A(s+1,z)\big)
\cdot A(\Phi_{s+1}(t,z))^{-1}
\equiv 
\frac{\der}{\der z_j}\big(\nabla A(s,z)\big) \cdot A(\Phi_{s}(t,z))^{-1}.
\eea
These congruences imply the theorem.
\end{proof}

\vsk.2>

\begin{cor}
\label{thm KZ mod p}
For $s\geq 1$ we have
\bea
I_{s}(z) \cdot A(\Phi_{s}(t,z))^{-1} \equiv 
I_{1}(z) \cdot A(\Phi_{1}(t,z))^{-1} \pmod{p}.
\eea

\end{cor}

\vsk.2>

\subsection{Domain of convergence}

By Theorem \ref{thm F ne} the polynomial $\det A(\Phi_1(t,z)) \in \Z[z]$ is  of degree $d_\Phi$ and
this polynomial is nonzero modulo $p$. For a positive integer $m$ define 
\bea
\frak D^{(m)}_{\on{KZ}}
= \{a \in (\Z_p^{(m)})^{n} \mid\  |\det A(\Phi_1(t,a))|_p=1\}\,.
\eea
By Lemma \ref{lem nonempty} the domain
$\frak D^{(m)}_{\on{KZ}}$ is nonempty if $p^m> d_\Phi$.
In what follows we assume that $p^m>d_\Phi$.

We have
$\vert\det A(\Phi_{s}(t,a))\big\vert_p =1$ for $ a\in \frak D^{(m)}_{\on{KZ}}$.
All entries of $A(\Phi_{s}(t,z))^{-1}$ are
rational functions in $z$ regular on $\frak D^{(m)}_{\on{KZ}}$.  
For every $a\in\frak D^{(m)}_{\on{KZ}}$  all entries of $A(\Phi_{s}(t,a))$ and $A(\Phi_{s}(t,a))^{-1}$ 
are elements of $\Z_p^{(m)}$.

\begin{thm}
\label{thm coKZ}
The sequence of $g\times g$ matrices 
\bea
\big(A\big(\Phi_{s}(t,z)\big)\cdot \si^e\big(A\big(\Phi_{s-1}(t,z)\big)\big)^{-1}\big)_{s\geq 1}\,,
\eea
whose entries are rational functions in $z$ regular on  $\frak D^{(m)}_{\on{KZ}}$,
 uniformly converges on $\frak D^{(m)}_{\on{KZ}}$ as $s\to\infty$ to an analytic
$g\times g$ matrix which will be denoted by $\mc A(z)$.
For $a\in\frak D^{(m)}_{\on{KZ}}$ we have
\bea
\big\vert \det \mc A(a) \big\vert_p=1\,
\eea
and the matrix $\mc A(a)$ is invertible. 
\end{thm}

\begin{proof}
The theorem follows from Theorem \ref{thm conv}.
\end{proof}

\begin{thm}
\label{thm coKZ}
For $i=1,\dots,n$ the sequence of $g\times g$ matrices 
\bea
\Big(\Big(\frac{\der}{\der z_i}A\big(\Phi_{s}(t,z)\big)\Big)
\cdot A\big(\Phi_{s}(t,z)\big)^{-1}\Big)_{s\geq 1}\,,
\eea
whose entries are rational functions in $z$ regular on  $\frak D^{(m)}_{\on{KZ}}$,
uniformly converges on $\frak D^{(m)}_{\on{KZ}}$ as $s\to\infty$ to an analytic
$g\times g$ matrix, which will be denoted by $\mc A^{(i)}(z)$.

The sequence of $n\times g$ matrices 
\bea
\big(I_s(z)\cdot A\big(\Phi_{s}(t,z)\big)^{-1}\big)_{s\geq 1}\,,
\eea
whose entries are rational functions in $z$ regular on  $\frak D^{(m)}_{\on{KZ}}$,
uniformly converges on $\frak D^{(m)}_{\on{KZ}}$ as $s\to\infty$ to an analytic
$n\times g$ matrix which will be denoted by $\mc I(z)$.

For $i=1,\dots, n$ the sequence of $n\times g$ matrices 
\bea
\Big(\frac{\der I_s}{\der z_i}(z)\cdot A\big(s,  \Phi_{s}(t,z)\big)^{-1}\Big)_{s\geq 1}\,,
\eea
whose entries are rational functions in $z$ regular on  $\frak D^{(m)}_{\on{KZ}}$,
uniformly converges on $\frak D^{(m)}_{\on{KZ}}$ as $s\to\infty$ to an analytic
$n\times g$ matrix which will be denoted by $\mc I^{(i)}(z)$.

We have
\bea
\frac{\der \mc I}{\der z_i} = \mc I^{(i)} - \mc I \cdot \mc A^{(i)}\,.
\eea

\end{thm}

\begin{proof}
The theorem follows from Theorems \ref{thm conv2},  \ref{thm conv3}, and Lemma \ref{lem conv3}.
\end{proof}

\begin{thm}
\label{thm KZ mc}

We have the following system of equations on $\frak D^{(m)}_{\on{KZ}}$\,:
\bea
\mc I^{(i)} = H_i \cdot \mc I, \qquad i=1,\dots, n,
\eea
where $H_i$ are the Gaudin Hamiltonians defined in \eqref{GH}.

\end{thm}

\begin{proof}

The theorem is a corollary of Theorem \ref{thm 7.3}.
\end{proof}

\begin{cor}
\label{cor mc I 1}

For $a\in \frak D^{(m)}_{\on{KZ}}$ we have
\bea
\mc I(a) \equiv I_1(a)\cdot A\big(\Phi_{1}(t,a)\big)^{-1}
\pmod{p}.
\eea

\end{cor}

\begin{proof} 

The corollary follows from Corollary \ref{thm KZ mod p}
 and Theorem \ref{thm coKZ}.
\end{proof}

\subsection{Vector bundle $\mc L  \,\to\,  \frak D^{(m),o}_{\on{KZ}}$} 

Denote 
\bea
W=\{(I_1,\dots,I_n)\in (\Q_p^{(m)})^n\ |\ I_1+\dots+I_n=0\}.
\eea
  We consider vectors
$(I_1,\dots,I_n)$ as column vectors. 
The differential operators $\nabla^{\on{KZ}}_i$, $i=1,\dots,n$, define a connection on
the trivial bundle $W\times \frak D^{(m)}_{\on{KZ}} \to \frak D^{(m)}_{\on{KZ}}$,
called the \KZ/ connection. The connection has singularities at the diagonal hyperplanes in $(\Z_p^{(m)})^n$
 and is well-defined over 

\bea
\frak D^{(m),o}_{\on{KZ}} = \{ a=(a_1,\dots,a_n)\in(\Z_p^{(m)})^n 
\mid |\det A(\Phi_1(t,a))|_p=1, \, a_i\ne a_j\,\ \forall i,j\}. 
\eea

\vsk.2>
\noindent
The \KZ/ connection is flat, 
\bea
\big[\nabla^{\on{KZ}}_i, \nabla^{\on{KZ}}_j\big]=0 \qquad \forall\,i,j,
\eea
see \cite{EFK}.
The flat sections of the \KZ/ connection are solutions of  system  \eqref{KZ} of \KZ/ equations.
 
\vsk.2>
For any $a\in \frak D^{(m)}_{\on{KZ}}$ let $\mc L_a \subset W$ be the vector subspace generated by columns of the 
$n\times g$ matrix $\mc I(a)$. Then
\bea
\mc L := \bigcup\nolimits_{a\in  \frak D^{(m)}_{\on{KZ}}}\,\mc L_a \,\to\,  \frak D^{(m)}_{\on{KZ}}
\eea
is an analytic distribution of vector subspaces  in the fibers of the trivial bundle
$W\times \frak D^{(m)}_{\on{KZ}} \to \frak D^{(m)}_{\on{KZ}}$.

\begin{thm}
[{\cite[Theorem 6.7]{VZ2}}]

\label{thm inv} 

The distribution $\mc L  \,\to\,  \frak D^{(m)}_{\on{KZ}}$   is invariant with respect to the \KZ/ connection.
In other words, if $s(z)$ is a local section of $\mc L  \,\to\,  \frak D^{(m)}_{\on{KZ}}$, then the
sections $\nabla_i^{\on{KZ}} s(z)$, $i=1,\dots,n$, also are sections of $\mc L  \,\to\,  \frak D^{(m)}_{\on{KZ}}$.

\end{thm}

\begin{proof} 
Let $\mc I(z)= (\mc I_1(z), \dots, \mc I_g(z))$ be columns of the $n\times g$ matrix 
$\mc I(z)$.
Let $a\in \frak D^{(m)}_{\on{KZ}}$. Let $c(z) = (c_1(z),\dots,c_g(z))$ be a column vector of analytic functions at $a$.
Consider a local section of the distribution  $\mc L  \,\to\,  \frak D^{(m)}_{\on{KZ}}$,
$
s(z)\, =\,  \sum_{j=1}^g \,c_j(z) \mc I_j(z)\,  = : \, \mc I \cdot c$.
Then
\bea
\nabla^{\on{KZ}}_i s(z) 
&=&
 - H_i\cdot \mc I\cdot c +\frac{\der \mc I}{\der z_i} \cdot c 
+ \mc I \cdot \frac{\der c}{\der z_i}
\\
&=&
 - H_i\cdot \mc I\cdot c +( \mc I^{(i)} - \mc I \cdot \mc A^{(i)})\cdot c 
+ \mc I \cdot \frac{\der c}{\der z_i}
\\
&=&
 - H_i\cdot \mc I\cdot c +( H_i \cdot \mc I  - \mc I \cdot \mc A^{(i)})\cdot c 
+ \mc I \cdot \frac{\der c}{\der z_i}
\\
&=&
  - \mc I \cdot \mc A^{(i)} \cdot c 
+ \mc I \cdot \frac{\der c}{\der z_i} \,.
\eea
Clearly, the last expression is a local section of $\mc L  \,\to\,  \frak D^{(m)}_{\on{KZ}}$.
\end{proof}

\begin{thm}

\label{thm rk}

The function $a \mapsto \dim_{\Q_p^{(m)}} \mc L_a$ is  constant on
$\frak D^{(m), o}_{\on{KZ}}$, in other words,  
$\mc L  \,\to\,  \frak D^{(m)}_{\on{KZ}}$  is a vector bundle over $\frak D^{(m),o}_{\on{KZ}}\subset \frak D^{(m)}_{\on{KZ}}$.

\end{thm}

The proof coincides with the proof of Theorem 6.8 in \cite{VZ2}.

\vsk.2>

Recall that $d_\Phi$ is the degree of the polynomial $\det A(\Phi_1(t,z))$ and
$d_M$ is the degree of the minor defined in Lemma \ref{lem minor}.

\begin{thm}
\label{thm rk g}

If $p^m> d_\Phi + d_M$, 
then the  analytic vector bundle $\mc L\to \frak D^{(m),o}_{\on{KZ}}$ is of rank $g$.

\end{thm}

\begin{proof}
If $p^m> d_\Phi + d_M$, then the minor $M(z)$ defines a function on 
$\frak D^{(m),o}_{\on{KZ}}$ nonzero modulo $p$ by  Lemma \ref{lem nonempty}.
Then by  Corollary \ref{cor mc I 1},
 the $n\times g$ matrix valued function
$\mc I(z)$ has a $g\times g$ minor nonzero on $\frak D^{(m)}_{\on{KZ}}$. 
This proves the theorem.
\end{proof}

\subsection{Remarks}

\subsubsection{} One may expect that the subbundle $\mc L\to \frak D^{(m),o}_{\on{KZ}}$
can be extended to a rank $g$ subbundle over 
$\frak D^{(m)}_{\on{KZ}} - \frak D^{(m),o}_{\on{KZ}}$, the union
 of the diagonal hyperplanes in 
$ \frak D^{(m)}_{\on{KZ}}$.

 \subsubsection{}
 Following Dwork we may expect that locally at any point $a\in\frak \frak D^{(m),o}_{\on{KZ}}$, the
 solutions of the KZ equations with values in $\mc L\to \frak D^{(m),o}_{\on{KZ}}$
 are given at $a$ by power series in 
$z_i-a_i$, $i=1,\dots,n$,  bounded in their polydiscs of convergence, 
while any other local solution at $a$ is given by a power series unbounded in its polydisc of convergence,
cf. \cite{Dw} and \cite[Theorem A.4]{V5}.

\subsubsection{}

The KZ connection $\nabla^{\on{KZ}}_i$, $i=1,\dots,n$, over $\C$ has no nontrivial proper invariant subbundles
due to the irreducibility of its monodromy representation, see \cite[Lemma 6]{Fo}.
 Thus the existence of the invariant subbundle $\mc L\to \frak D^{(m),o}_{\on{KZ}}$
 is a $p$-adic feature.

\subsubsection{}
The invariant subbundles of the KZ connection over $\C$ usually are related to some additional conformal block
constructions, for example see \cite{FSV, SV2, V3,  V4}. Apparently our subbundle  $\mc L\to \frak D^{(m),o}_{\on{KZ}}$
 is of a
different $p$-adic nature.

\end{document}